\documentclass[11pt]{amsart}
\usepackage[a4paper,centering,left=2.2cm,right=2.2cm]{geometry}
\usepackage{amsfonts,amscd,amssymb,amsmath,amsthm,hyperref,mathrsfs,multirow,xcolor,lscape,longtable,pbox,lipsum,afterpage,graphicx,threeparttable}
\usepackage[thinlines]{easytable}
\usepackage{microtype}

\newtheorem{thm}{Theorem}

\newtheorem{lema}[thm]{Lemma}

\newtheorem{coro}[thm]{Corollary}
\newtheorem{rmk}[thm]{Remark}

\newtheorem*{thm*}{Theorem}

\usepackage[all,cmtip]{xy}
\usepackage{tikz}
\usetikzlibrary{arrows} 
\usetikzlibrary{matrix}
\usetikzlibrary{calc}

\newcommand{\R}{\mathbb{R}}
\newcommand{\Z}{\mathbb{Z}}
\newcommand{\Q}{\mathbb{Q}}

\newcommand{\G}{\Gamma}
\newcommand{\N}{\mathbb{N}}
\newcommand{\T}{\mathrm{T}}

\newcommand{\h}{\mathbb{H}}
\newcommand{\rH}{\mathrm{H}}

\newcommand{\m}{\mathcal{M}}

\newcommand{\rG}{\mathrm{G}}
\newcommand{\rS}{\mathrm{S}}

\newcommand{\rP}{\mathrm{P}}

\newcommand{\W}{\mathcal{W}}

\newcommand{\tmt}[4]{\left({#1\atop #3}{#2\atop #4}\right)}

\newcommand{\eqr}[1]{\mbox{(\ref{eq:#1})}}
\newcommand{\ie}{i.e.\ }
\newcommand{\mr}[1]{\mathrm{#1}}
\newcommand{\C}{\mathbb{C}}

\newcommand{\tm}{\widetilde{\m}}
\newcommand{\mc}[1]{\mathcal{#1}}
\newcommand{\mcV}{\mathcal{V}}

\begin{document}

\title[Boundary and Eisenstein Cohomology]{Boundary and  Eisenstein  Cohomology of $\mr{SL}_3(\Z)$} 
\date{\today}
\author{Jitendra Bajpai, G\"unter Harder, Ivan Horozov  and Matias Moya Giusti }
\address{\tiny{Mathematisches Institut, Georg-August Universit\"at G\"ottingen, Bunsenstrasse 3-5, D-37073 Germany}}
\email{jitendra@math.uni-goettingen.de}
\address{\tiny{Max Planck Institute for Mathematics, Vivatsgasse~7, D-53111, Bonn, Germany}}
\email{harder@mpim-bonn.mpg.de}
\address{\tiny{Graduate Center of the City University of New York (CUNY), 365 5th Ave, NY 10016, USA}}
\address{and}
\address{\tiny{Department of Mathematics and Computer Science, Bronx Community College, CUNY, 2155 University Ave, NY 10453, USA.}}
\email{ivan.horozov@bcc.cuny.edu}
\address{\tiny{Universit\'e Paris-Est Marne-la-Vall\'ee, 5 Boulevard Descartes, 77454 Marne-la-Vall\'ee, France}}
\email{matias-victor.moyagiusti@u-pem.fr}


\begin{abstract}

In this article, several cohomology spaces associated to the arithmetic groups $\mr{SL}_3(\Z)$ and $\mr{GL}_3(\Z)$ with coefficients in any highest weight representation $\m_\lambda$ have been computed, where $\lambda$ denotes their highest weight.  Consequently, we obtain detailed information of their Eisenstein cohomology  with coefficients in $\m_\lambda$. When $\m_\lambda$ is not self dual, the Eisenstein cohomology coincides with the cohomology of the underlying arithmetic group with coefficients in $\m_\lambda$. In particular, for such a large class of representations we can explicitly describe the cohomology of these two arithmetic groups. We accomplish this by studying the cohomology of the boundary of the Borel-Serre compactification and their Euler characteristic with coefficients in $\m_\lambda$. At the end, we employ our study to discuss the existence of ghost classes.

\end{abstract}

\maketitle
\tableofcontents




\section{Introduction}
Let $\mr{G}$ be a split semisimple group defined over $\mathbb{Q}$, then for every arithmetic subgroup $\Gamma \subset \mr{G}(\mathbb{Q})$ one can define the corresponding locally symmetric space
\[
	\rS_\Gamma = \Gamma \backslash \mr{G}(\mathbb{R}) / \mr{K}_\infty
\]
where $\mr{K}_\infty$ denotes the maximal connected compact subgroup of $\rG(\mathbb{R})$. In this context we can consider the Borel-Serre compactification $\overline{\rS}_\Gamma$ of $\rS_\Gamma$ (see \cite{BoSe73}), whose boundary $\partial \rS_\Gamma$ is a union of spaces indexed by the $\Gamma$-conjugacy classes of $\mathbb{Q}$-parabolic subgroups of $\rG$. For the detailed account on Borel-Serre compactification, see~\cite{Harder2018}.
The choice of a maximal $\mathbb{Q}$-split torus $\T$ of $\rG$ and a system of positive roots $\Phi^+$ in $\Phi(\rG, \T)$ determines a set of representatives for the conjugacy classes of $\mathbb{Q}$-parabolic subgroups, namely the standard $\mathbb{Q}$-parabolic subgroups. We will denote this set by $\mathcal{P}_\mathbb{Q}(\rG)$.  
One can write the boundary $\partial \rS_\Gamma$ as a union
\begin{equation}\label{eq:covering}
\partial \rS_\G = \bigcup_{\mr{P} \in \mathcal{P}_\mathbb{Q}(\rG)} \partial_{\mr{P}, \G}\,.
\end{equation}
The irreducible representation $\m_\lambda$ of $\rG$ associated to a highest weight $\lambda$ defines a sheaf over $\rS_\G$, denoted by $\tm_\lambda$, that is defined over $\Q$. This sheaf can be extended in a natural way to a sheaf in the Borel-Serre compactification $\overline{\rS}_\G$ and we can therefore consider the restriction to the boundary of the Borel-Serre compactification and to each face of the boundary, obtaining sheaves in $\partial \rS_\Gamma$ and $\partial_{\mr{P}, \G}$. 
The aforementioned covering defines a spectral sequence abutting to the cohomology of the boundary
\begin{equation}\label{eq:spec}
E_1^{p, q} = \bigoplus_{prk(\mr{P}) = p+1} H^q(\partial_{\mr{P}, \Gamma}, \tm_\lambda) \Longrightarrow H^{p+q}(\partial \rS_\Gamma, \tm_\lambda).
\end{equation}
where $prk(\mr{P})$ denotes the parabolic rank of $\mr{P}$ (the dimension of its $\mathbb{Q}$-split component). In this article we present an explicit description of this spectral sequence to discuss in detail the boundary and Eisenstein cohomology for the particular rank two cases $\mr{SL}_3$ and $\mr{GL}_3$.

\par  Since its development, cohomology of arithmetic groups has been proved to be a valuable tool  in analyzing the  relations between the theory of automorphic forms and the arithmetic properties of the associated locally symmetric spaces. A very common goal  is to describe the cohomology $H^{\bullet}(\mr{S}_{\Gamma}, \tm_{\lambda})$ in terms of automorphic forms. The study of boundary and Eisenstein cohomology of arithmetic groups has many number theoretic applications. As an example, one can see applications on the algebraicity of certain quotients of special values of $L$-functions in~\cite{Harder87}.

 \par The main tools and idea to study the boundary cohomology of  arithmetic groups have been developed by the second author in a series of articles~\cite{Harder87, Harder90, Harder2012}. This article is no exception in taking the hunt a little further. Especially, we make use of the techniques developed in~\cite{Harder2012}. In a way, this article is a continuation of the work carried out by the second author in~\cite{Harder90}. In Section 4, the cohomology of the boundary of $\mr{SL}_3(\Z)$ has been described after introducing the necessary notations and tools in Section 2 and Section 3.

\par In order to achieve the details about the space of Eisenstein cohomology of the two mentioned arithmetic groups, we make use of their Euler characteristics. In Section 5, we discuss this in detail. The importance of Euler charcateristic to study the space of Eisenstein cohomology has been discussed by the third author in~\cite{Horozov2014}. For more details about Euler characteristic of arithmetic groups see~\cite{Horozov2004, Horozov2005}. In Section 6, we compute the space of Eisenstein cohomology of the arithmetic groups $\mr{SL}_3(\Z)$ and $\mr{GL}_3(\Z)$ with coefficients in $\m_{\lambda}$. One of the most interesting take aways, among others, of these two sections is the intricate relation between the spaces of automorphic forms of $\mr{SL}_{2}$  and the boundary and Eisenstein cohomology spaces of $\mr{SL}_3$.

\par In Section~\ref{ghost}, we carry out the discussion of existence of ghost classes in $\mr{SL}_3(\Z)$ and $\mr{GL}_3(\Z)$ in detail with respect to any highest weight representation. Ghost classes were introduced by A. Borel~\cite{Borel84} in 1984. For details and exact definition of these classes see Section~\ref{ghost}. Later on, these classes have appeared in the work of the second author. For example at the end of the article~\cite{Harder90} with emphasis to the case $\mr{GL}_3$, it is mentioned that  \emph{ ``.... the ghost classes appear if some $L$-values vanish. The order of vanishing does not play a role. But this may change in higher rank case"}. The author further added that this aspect is worthy of investigation. The importance of their investigation has been occasionally pointed out. Since then, these classes have been studied at times, however the general theory of these classes has been slow in coming. We couldn't trace down the complete analysis of ghost classes in these two specific cases in complete generality, \ie for arbitrary coefficient system. However, in case of $\mr{SL}_4(\Z)$ these classes have been discussed by Rohlfs in~\cite{Rohlfs96}. In general for $\mr{SL}_n$, Franke developed a method to construct ghost classes in~\cite{Franke81}. Later on, using the method developed in~\cite{Franke81}, Kewenig and Rieband have found ghost classes for the orthogonal and symplectic groups when the coefficient system is trivial, see~\cite{KeRi97}. More recently, these classes have been discussed by the first and last author in the case of rank two orthogonal Shimura varieties in~\cite{BMG18} and by the last author in case of $\mr{GSp}_{4}$ in~\cite{Moya17} and $\mr{GU}(2,2)$ in~\cite{Moya}.

\par The main results of this article are the following,
\begin{itemize}
\item Theorem~\ref{euler}, where the Euler characteristic of $\mr{SL}_3(\Z)$ is calculated with respect to every finite dimensional highest weight representation. 
\item Theorem~\ref{bdcohsl3}, where the boundary cohomology with coefficients in every finite dimensional highest weight representation is described. 
\item Theorem~\ref{thm 1/2}, that shows that the Euler characteristic of the boundary cohomology is half the Euler characteristic of the Eisenstein cohomology. 
\item Theorem~\ref{Eiscoh}, where we describe the Eisenstein cohomology for every finite dimensional highest weight representation.
\item Theorem~\ref{ghostthm}, that shows that there are no ghost classes unless possibly in degree two for certain nonregular highest weights.
\end{itemize} 

In this paper we do not refer to and do not use transcendental methods, i.e. we do not write down
convergent (or even non convergent) infinite series and do not use the principle of analytic continuation.
This allows us to work with coefficient systems which are $\Q$-vector spaces. Only at one place we refer
to the Eichler-Shimura isomorphism, but this reference is not really relevant.  At one point we refer to
a deep theorem  of Bass-Milnor-Serre~\cite{BMS67} to get the complete description of the  Eisenstein cohomology.
Transcendental arguments would allow us to avoid  this reference, see~\cite{Harder2010} and~\cite{Sch83}.

In Theorem~~\ref{ghostthm} we leave open, whether in a certain case ghost classes might  exist.  In a letter to A.~Goncharov the second author has outlined an argument that shows that there are no ghost classes, but this argument depends on transcendental methods. This will be discussed in a forthcoming paper.



\section{Basic Notions}\label{basic}
This section provides quick review to the basic properties of $\mr{SL}_3$ (and $\mr{GL}_3$) and familiarize the reader with the notations to be used throughout the article. We discuss the corresponding locally symmetric space, Weyl group, the associated spectral sequence and Kostant representatives of the standard parabolic subgroups.

\subsection{Structure theory}\label{maximal}  
Let $\T$ be the maximal torus of $\mr{SL}_3$ given by the group of diagonal matrices and $\Phi$ be the corresponding root system of type $\mr{A}_2$. Let $\epsilon_1, \epsilon_2, \epsilon_3 \in X^\ast(\T)$ be the usual coordinate functions on $\T$. We will use the additive notation  for the abelian group $X^\ast(\T)$ of characters of $\T$. The root system is given by $\Phi=\Phi^{+} \cup \Phi^{-}$, where $\Phi^{+}$ and $\Phi^{-}$ denote the set of positive and negative roots of $\mr{SL}_3$ respectively, and $\Phi^{+} = \left\{  \epsilon_1-\epsilon_3, \epsilon_1-\epsilon_2, \epsilon_2-\epsilon_3 \right\}$. Then the system of simple roots is defined by $\Delta=\left\{\alpha_1=\epsilon_1-\epsilon_2, \alpha_2=\epsilon_2- \epsilon_3\right\}$. The fundamental weights associated to this root system are given by $\gamma_1 = \epsilon_1$ and $\gamma_2=\epsilon_1+\epsilon_2$. The irreducible finite dimensional representations of $\mr{SL}_3$ are determined by their highest weight which in this case are the elements of the form $\lambda = m_1 \gamma_1 + m_2 \gamma_2$ with $m_1, m_2$ non-negative integers. The Weyl group $\mathcal{W}$ of $\Phi$ is given by the symmetric group  $\mathfrak{S}_3$.

The above defined root system determines a set of proper standard $\mathbb{Q}$-parabolic subgroups $\mathcal{P}_\mathbb{Q}(\mr{SL}_3)=\left\{ \mr{P}_0, \mr{P}_1, \mr{P}_2 \right\}$, where $\mr{P}_0$ is a minimal and $\mr{P}_1, \mr{P}_2$ are maximal $\mathbb{Q}$-parabolic subgroups of $\mr{SL}_3$. To be more precise, we write
\[ 
\mr{P}_1(A) = \left\{ \left( \begin{array}{cccccc}
\ast & \ast & \ast  \\
0 & \ast & \ast   \\
0 & \ast & \ast \\ 
\end{array}  \right) \in \mr{SL}_3(A) \right\} \,, \quad
\mr{P}_2(A) = \left\{ \left( \begin{array}{ccccccc}
\ast & \ast & \ast \\
\ast & \ast & \ast  \\
0 & 0 & \ast  \\ 
 \end{array}  \right) \in \mr{SL}_3(A) \right\} \,,
\]
for every $\Q$-algebra $A$, and  $\mr{P}_0$ is simply the group given by $\rP_1 \cap \rP_2$.

The set $\mathcal{P}_\mathbb{Q}(\mr{SL}_3)$ is a set of representatives for the conjugacy classes of $\mathbb{Q}$-parabolic subgroups of $\mr{SL}_3$. Consider the maximal connected compact subgroup $\mr{K}_\infty = \mr{SO}(3) \subset \mr{SL}_3(\mathbb{R})$ and the arithmetic subgroup $\Gamma = \mr{SL}_3(\mathbb{Z})$, then $\mr{S}_\Gamma$ denotes the orbifold 
$ \Gamma \backslash \mr{SL}_3(\mathbb{R})/\mr{K}_\infty.$ Note that in terms of differential geometry $\mr{S}_\Gamma$ is not a locally symmetric space, this is because of the torsion elements in $\Gamma$. 

\subsection{Spectral sequence}

Let $\overline{\rS}_\Gamma$ denote the Borel-Serre compactification of $\rS_\Gamma$ (see \cite{BoSe73}). Following~\eqr{covering}, the boundary of this compactification $\partial \rS_\Gamma = \overline{\rS}_\Gamma \setminus \rS_\Gamma$ is given by the union of faces indexed by the $\Gamma$-conjugacy classes of $\mathbb{Q}$-parabolic subgroups. Consider the irreducible representation $\m_\lambda$ of $\mr{SL}_3$ associated with a highest weight $\lambda$. This representation is defined over $\Q$ and determines a sheaf $\tm_\lambda$ over $\rS_\Gamma$. By applying the direct image functor associated to the inclusion $i:\rS_\Gamma \hookrightarrow \overline{\rS}_\Gamma$, we obtain a sheaf on $\overline{\rS}_\Gamma$ and, since this inclusion is a homotopy equivalence (see~\cite{BoSe73}), it induces an isomorphism
$H^\bullet(\rS_\Gamma, \tm_\lambda) \cong H^\bullet(\overline{\rS}_\Gamma, i_\ast(\tm_\lambda)).$ From now on $i_\ast(\tm_\lambda)$ will be simply denoted by $\tm_\lambda$. In this paper, one of our immediate goals is to  make a thorough study of the cohomology space of the boundary 
$H^\bullet(\partial \rS_\Gamma, \tm_\lambda).$

The covering \eqr{covering} defines a spectral sequence in cohomology abutting to the cohomology of the boundary. To be more precise, one has the spectral sequence defined by~\eqr{spec} in the previous section. To be able to study this spectral sequence, we need to understand the cohomology spaces $H^q(\partial_{\mr{P}, \Gamma}, \tm_\lambda)$ and this can be done by making use of a certain decomposition. To present the aforementioned decomposition we need to introduce some notations.

Let $\mr{P} \in \mathcal{P}_\mathbb{Q}(\mr{SL}_3)$ be a standard $\mathbb{Q}$-parabolic subgroup and $\mr{M}$ be the corresponding Levi quotient, then $\Gamma_\mr{M}$ and $\mr{K}_\infty^\mr{M}$ will denote the image under the canonical projection $\pi:\mr{P} \longrightarrow \mr{M}$ of the groups $\Gamma \cap \mr{P}(\mathbb{Q})$ and $\mr{K}_\infty \cap \mr{P}(\mathbb{R})$, respectively. $\left.^\circ \mr{M} \right.$ will denote the group
\[
	\left.^\circ \mr{M} \right.= \bigcap_{\chi \in X^\ast_\mathbb{Q}(\mr{M})} \chi^2
\]
where $X^\ast_\mathbb{Q}(\mr{M})$ denotes the set of $\mathbb{Q}$-characters of $\mr{M}$. Then $\Gamma_\mr{M}$ and  $\mr{K}_\infty^\mr{M}$ are contained in $^\circ \mr{M}(\mathbb{R})$ and we define the locally symmetric space of the Levi quotient $\mr{M}$ by
\[
	\rS_\Gamma^\mr{M} = \Gamma_\mr{M} \backslash ^\circ \mr{M}(\mathbb{R})/\mr{K}_\infty^\mr{M}.
\]

On the other hand, let $$\mathcal{W}^{\rP}=\{ w \in \mc{W} | w(\Phi^{-}) \cap \Phi^{+} \subseteq \Phi^{+}(\mathfrak{n})\}$$ be the set of Weyl representatives of the parabolic $\rP$ (see \cite{Kostant61}), where $\mathfrak{n}$ is the Lie algebra of the unipotent radical of $\mr{P}$ and $\Phi^{+}(\mathfrak{n})$ denotes the set of roots whose root space is contained in $\mathfrak{n}$. If $\rho \in X^\ast(\mr{T})$ denotes half of the sum of the positive roots (in this case this is just $\epsilon_1 - \epsilon_3$) and $w \in \mathcal{W}^{\mr{P}}$, then the element $w\cdot\lambda = w(\lambda + \rho) - \rho$ is a highest weight of an irreducible representation $\mathcal{M}_{w\cdot\lambda}$ of $\left. ^\circ \mr{M} \right.$ and defines a sheaf $\widetilde{\m}_{w\cdot\lambda}$ over $\rS_\Gamma^\mr{M}$. Then we have a decomposition
\[
	H^q(\partial_{\mr{P}, \Gamma}, \tm_\lambda) = \bigoplus_{w \in \mathcal{W}^{\mr{P}}} H^{q-\ell(w)}(\rS_\Gamma^\mr{M}, \tm_{w\cdot\lambda}).
\]

\subsection{Kostant Representatives of Standard Parabolics}

In the next table we list all the elements of the Weyl group along with their lengths and the preimages of the simple roots. The preimages will be useful to determine the sets of Weyl representatives for each parabolic subgroup.
\begin{table}[htbp]\label{weights4}
\label{table3}
\centering
\begin{tabular}{ccccccccc}
\hline\hline \noalign{\smallskip}
\noindent $w$ & $\sigma$  & $\ell(w)$ & $\sigma^{-1}(\alpha_1)$ & $\sigma^{-1}(\alpha_2)$  \\
\noalign{\smallskip}\hline\hline\noalign{\smallskip}
$e$ & $e$& $0$ &  $\alpha_1$  & $\alpha_2$ \\ 
$s_{1}$ & $(1\, 2)$ & $1$ & $-\alpha_1$ & $\alpha_1+\alpha_2$ \\ 
$s_{2}$ & $(2 \,3)$ & $1$ & $\alpha_1+\alpha_2$ & $-\alpha_2$ \\ 
$s_2 s_{1}$ & $(1\, 3 \,2)$  & $2$ &  $\alpha_2$  & $-\alpha_1-\alpha_2$ \\ 
$s_{1}s_2$ & $(1\, 2 \,3)$  & $2$ & $-\alpha_1-\alpha_2$ & $\alpha_1$ \\
$s_2s_{1}s_2$& $(1\, 3)$  & $3$ & $-\alpha_2$ & $-\alpha_1$ \\ 
(=$s_1s_{2}s_1$) & & & &\\
\noalign{\smallskip}\hline\hline\noalign{\smallskip}
\end{tabular}
\vspace{0.4cm}
\caption{The set of Weyl representatives $\mathcal{W}^{\mr{P}_0}$}
\end{table}
 
\begin{figure}
\centering
\begin{tikzpicture}[scale=2.8]
\draw[draw=red,-triangle 90] (0.25,0.5) node[above]{$e$} --(1, 0) node[below]{$(2\, 3)$};
\draw[draw=blue,-triangle 90] (0.25,0.5) -- (-0.5,0)node[below]{$(1\, 2)$}; 
\draw[draw=blue,-triangle 90]  (1, -0.15) -- (1,-0.7) node[below]{$(1\, 2 \, 3)$} ; 
\draw[draw=red,-triangle 90,fill=blue] (-0.5, -0.15) -- (-0.5,-0.7) node[below]{$(1\, 3\, 2)$};
\draw[draw=red,-triangle 90] (1,-0.87) -- (0.25,-1.30) node[below]{$(1\, 3)$};
\draw[draw=blue,-triangle 90] (-0.5, -0.87)--(0.25, -1.30);
\draw[draw=red,-triangle 90,fill=blue];
 \end{tikzpicture}
\caption{ Weyl elements : blue and red arrows denote the simple reflections $s_1$ and $s_2$ respectively. $s_i$ acts from the left on the element that appears on the tale and the result of this action appears at the head of the respective arrow.}\label{t3}
\end{figure}
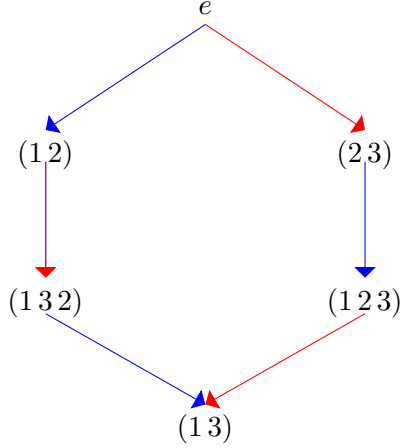

Note that in the case of $\mr{SL}_3$, $\epsilon_3 = -(\epsilon_1+\epsilon_2)$ and $\mathcal{W}^{\mr{P}_0} = \mathcal{W}$. Now,  by using this table, one can see that the sets of Weyl representatives for the maximal parabolics $\mr{P_1}$ and $\mr{P}_{2}$, are given by $$\mathcal{W}^{\mr{P}_1} = \left\{ e, s_1, s_1s_2 \right\} \quad \mr{and} \quad \mathcal{W}^{\mr{P}_2} = \left\{ e, s_2, s_2s_1 \right\}.$$ We now record for each standard parabolic $\mr{P}$ and Weyl representative $w \in \mathcal{W}^\mr{P}$, the expression $w\cdot\lambda$ in the convenient setting so that it can be used to obtain Lemma~\ref{parity0} and Lemma~\ref{parity1} which commence in the next few pages. Let $\lambda$ be given by $m_1\gamma_1 + m_2\gamma_2$, then the Kostant representatives for parabolics $\mr{P}_0, \mr{P}_1$ and $\mr{P}_2$ are listed respectively, where we make use of the notations
\begin{align}
\gamma^{\mr{M}_1} = \frac{1}{2}(\epsilon_2 - \epsilon_3), \quad &\kappa^{\mr{M}_1} = \frac{1}{2}(\epsilon_2 + \epsilon_3), \nonumber \\
\gamma^{\mr{M}_2} = \frac{1}{2}(\epsilon_1 - \epsilon_2), \quad &\kappa^{\mr{M}_2} = \frac{1}{2}(\epsilon_1 + \epsilon_2). \nonumber
\end{align}

\subsubsection{Kostant representatives for minimal parabolic $\mr{P}_0$}
\begin{align} \label{highest0}
e\cdot\lambda &= m_1\gamma_1 + m_2\gamma_2 \nonumber \\
s_1\cdot\lambda &= (-m_1-2) \gamma_1 + (m_1+m_2+1) \gamma_2 \nonumber \\
s_2\cdot\lambda &=  (m_1+m_2+1) \gamma_1 + (-m_2-2) \gamma_2 \nonumber \\
s_1s_2\cdot\lambda &= (-m_1-m_2-3) \gamma_1 + m_1 \gamma_2 \nonumber \\
s_2s_1\cdot\lambda &=  m_2 \gamma_1 + (-m_1-m_2-3) \gamma_2 \nonumber \\
s_2s_{1}s_2 \cdot \lambda = s_1s_2s_1\cdot\lambda &=  (-m_2-2) \gamma_1 + (-m_1-2) \gamma_2 \nonumber 
\end{align}

\subsubsection{Kostant representatives for maximal parabolic $\mr{P}_1$}
\begin{align} 
e\cdot\lambda &= m_2\gamma^{\mr{M}_1} + (-2m_1-m_2) \kappa^{\mr{M}_1} \nonumber \\
s_1\cdot\lambda &= (m_1+m_2+1) \gamma^{\mr{M}_1} + (m_1-m_2+3)\kappa^{\mr{M}_1} \nonumber \\
s_1s_2\cdot\lambda &= m_1 \gamma^{\mr{M}_1} + (m_1+2m_2+6) \kappa^{\mr{M}_1} \nonumber
\end{align}

\subsubsection{Kostant representatives for maximal parabolic $\mr{P}_2$}
\begin{align}
e\cdot\lambda &= m_1 \gamma^{\mr{M}_2} + (m_1+2m_2)\kappa^{\mr{M}_2} \nonumber \\
s_2\cdot\lambda &=  (m_1+m_2+1) \gamma^{\mr{M}_2} + (m_1-m_2-3) \kappa^{\mr{M}_2} \nonumber \\
s_2s_1\cdot\lambda &=  m_2 \gamma^{\mr{M}_2} + (-2m_1-m_2-6) \kappa^{\mr{M}_2} \nonumber
\end{align}




\section{Parity Conditions in Cohomology}\label{parity}

The cohomology of the boundary can be obtained by using a spectral sequence whose terms are given by the cohomology of the faces associated to each standard parabolic subgroup. In this section we expose, for each standard parabolic $\rP$ and irreducible representation $\mathcal{M}_{\nu}$ of the Levi subgroup $\mr{M} \subset \rP$ with highest weight $\nu$, a parity condition to be satisfied in order to have nontrivial cohomology $H^\bullet(\rS_\Gamma^\mr{M}, \widetilde{\mathcal{M}}_{\nu})$. Here $\rS_\G^\mr{M}$ denotes the symmetric space associated to $\mr{M}$ and $\tm_{\nu}$ is the sheaf in $\rS_\Gamma^\mr{M}$ determined by $\m_{\nu}$.

\subsection{Borel subgroup}
 
We begin by studying the parity condition imposed on the face associated to the minimal parabolic $\rP_0$ of $\mr{SL}_3$. The Levi subgroup of $\rP_0$ is the two dimensional torus $\mr{M}_0 = \mr{T}$ of diagonal matrices. To get nontrivial cohomology the finite group $ \Gamma_{\rm{M}_0}\cap \mr{K}^{\rm{M}_0}_\infty $ has to act trivially on $\m_{\nu}$, because  otherwise $\tm_{\nu}=0$. Therefore, the following three elements 
 
\[ 
\left( \begin{array}{rrr}
-1 & 0 & 0  \\
0 & -1 & 0  \\
0 & 0 & 1   \\ 
\end{array}  \right),
\left( \begin{array}{rrr}
-1 & 0 & 0  \\
0 & 1 & 0  \\
0 & 0 & -1   \\ 
\end{array}  \right),
\left( \begin{array}{rrr}
1 & 0 & 0  \\
0 & -1 & 0  \\
0 & 0 & -1   \\ 
\end{array}  \right) \in \Gamma_{\rm{M}_0}\cap \mr{K}^{\rm{M}_0}_\infty
\]
 must act trivially on $\m_{\nu}$ so that the sheaf $\tm_{\nu}$ is nonzero. By using this fact one can deduce the following 

\begin{lema} \label{parity0}
Let $\nu$ be given by $m'_1\gamma_1 + m'_2\gamma_2$. If $m'_1$ or $m'_2$ is odd then the corresponding local system $\tm_{\nu}$ in $\rS_\G^{\rm{M}_0}$ is $0$.  
\end{lema}

Note that the $\nu$ to be considered in this paper will be of the form $w\cdot \lambda$, for $w\in \W^{\mathrm{P}_0}$. We denote by $\overline{\W}^{0}(\lambda)$ the set of Weyl elements $w$ such that $w\cdot \lambda$ do not satisfy the condition of Lemma~\ref{parity0}.

\begin{rmk}
For notational convenience, we simply use $\partial_{i}$ to denote the boundary face $\partial_{\rP_{i}, \Gamma}$ associated to the parabolic subgroup $\rP_{i}$ and the arithmetic group $\Gamma$ for $i\in\{0,1,2\}$. In addition, we will drop the use of $\Gamma$ from the  $\rS_{\Gamma}$ and $\partial \rS_{\Gamma}$ and likewise from the other notations. 
\end{rmk}

\subsubsection{Cohomology of the face $\partial_0$} In this case $H^{q}(\rS^{\rm{M}_0}, \widetilde{\m}_{w\cdot\lambda})=0$ for every $q\geq 1$. The set of Weyl representatives $\W^{\rP_{0}}= \mathcal{W}$ and the lengths of its  elements  are between 0 and 3 as shown in the table and figure above. We know
\begin{eqnarray*}
 H^{q}(\partial_0, \tm_{\lambda}) &= &\bigoplus_{w\in  \mathcal{W}^{\rP_{0}}} H^{q-\ell(w)}(\rS^{\rm{M}_{0}}, \widetilde{\m}_{w\cdot\lambda})\\
 &=& \bigoplus_{w\in  \mathcal{W}^{\rP_{0}} : \ell(w)=q} H^{0}(\rS^{\rm{M}_{0}}, \widetilde{\m}_{w\cdot\lambda})  \nonumber
 \end{eqnarray*}
 Therefore
 \begin{eqnarray}
 H^{0}(\partial_0, \tm_{\lambda}) &= & H^{0}(\rS^{\rm{M}_{0}}, \widetilde{\m}_\lambda)\nonumber\\
 H^{1}(\partial_0, \tm_{\lambda}) &= & H^{0}(\rS^{\rm{M}_{0}}, \widetilde{\m}_{s_1\cdot\lambda}) \oplus H^{0}(\rS^{\rm{M}_{0}}, \widetilde{\m}_{s_2\cdot\lambda}) \nonumber\\
H^{2}(\partial_0, \tm_{\lambda}) &= & H^{0}(\rS^{\rm{M}_{0}}, \widetilde{\m}_{s_1s_2\cdot\lambda}) \oplus H^{0}(\rS^{\rm{M}_{0}}, \widetilde{\m}_{s_2s_1\cdot\lambda}) \nonumber\\
H^{3}(\partial_0, \tm_{\lambda}) &= & H^{0}(\rS^{\rm{M}_{0}}, \widetilde{\m}_{s_1s_2s_1\cdot\lambda}) \nonumber
\end{eqnarray}
 and for every $q \geq 4$, the cohomology groups $H^{q}(\partial_0, \tm_{\lambda})=0$.

\subsection{Maximal parabolic subgroups}

In this section we study the parity conditions for the maximal parabolics. Let $i \in \left\{1, 2\right\}$, then $\rm{M}_i \cong \mr{GL}_2$ and in this setting, $\mr{K}_\infty^{\rm{M}_i}=\mr{O}(2)$ is the orthogonal group and $\G_{\rm{M}_i} = \mr{GL}_2(\Z)$. Therefore 
\[
	\mr{S}^{\rm{M}_i} \cong \widetilde{\mr{S}}^{\mr{GL}_2} = \mr{GL}_2(\mathbb{Z}) \backslash \mr{GL}_2(\mathbb{R})/\mr{O}(2) \mathbb{R}_{>0}^\times
\]

Let $\epsilon'_1, \epsilon'_2$ denote the usual characters in the torus T of diagonal matrices of $\mr{GL}_2$. Write $\gamma = \frac{1}{2}(\epsilon'_1 - \epsilon'_2)$ and $\kappa = \frac{1}{2}(\epsilon'_1 + \epsilon'_2)$. Consider the irreducible representation $\mathcal{V}_{a, n}$ of $\mr{GL}_2$ with highest weight $a \gamma + n \kappa$. In this expression $a$ and $n$ must be congruent modulo $2$, and $\mathcal{V}_{a, n} = Sym^a(\mathbb{Q}^2) \otimes det^{(n-a)/2}$ is the tensor product of the $a$-th symmetric power of the standard representation and the determinant to the $(\frac{n-a}{2})$-th power. This representation defines a sheaf $\widetilde{\mathcal{V}}_{a, n}$ in $\widetilde{\mr{S}}^{\mr{GL}_2}$ and also in the locally symmetric space
\[
	\mr{S}^{\mr{GL}_2} =\mr{GL}_2(\mathbb{Z}) \backslash  \mr{GL}_2(\mathbb{R})/\mr{SO}(2)\mathbb{R}_{>0}^\times .
\]

If $\mr{Z} \subset \mr{T}$ denotes the center of GL$_2$, one has
\[ 
\left( \begin{array}{cc}
-1 & 0 \\
0 & -1 \\ 
\end{array}  \right) \in \mr{GL}_2(\mathbb{Z}) \cap \mr{Z}(\mathbb{R}) \cap \mr{SO}(2)\mathbb{R}_{>0}^\times
\]
and therefore this element must act trivially on $\mathcal{V}_{a, n}$ in order to have $\widetilde{\mathcal{V}}_{a, n} \neq 0$, i.e. if $n$ is odd then $\widetilde{\mathcal{V}}_{a, n} = 0$. So, we are just interested in the case in which $n$ (and therefore $a$) is even. On the other hand, if $a=0$, $\mathcal{V}_{a, n}$ is one dimensional and 
\[ 
\left( \begin{array}{cc}
-1 & 0 \\
0 & 1 \\ 
\end{array}  \right) \in \mr{GL}_2(\mathbb{Z}) \cap \mr{O}(2)\mathbb{R}_{>0}^\times
\]
has the effect that the space of global sections of $\widetilde{\mathcal{V}}_{0, n}$ is $0$ when $n/2$ is odd.

We summarize the above discussion in the following

\begin{lema}\label{parity1}
Let $i$ be $1$ or $2$. For $w \in \mathcal{W}^{\rP_i}$, let $w\cdot \lambda$ be given by $a \gamma^{\mr{M}_i} + n \kappa^{\mr{M}_i}$, where $\gamma^{\mr{M}_i} = \frac{1}{2}(\epsilon_i - \epsilon_{i+1})$ and $\kappa^{\mr{M}_i} = \frac{1}{2}(\epsilon_i + \epsilon_{i+1})$. If $n$ is odd, the corresponding sheaf $\tm_{w\cdot\lambda}$ is $0$. As $a$ and $n$ are congruent modulo $2$, we should have $a$ and $n$ even in order to have a non trivial coefficient system $\widetilde{\mathcal{V}}_{a, n}$. Moreover, if $a = 0$ and $n/2$ is odd, then $H^\bullet(\rS^{\mr{M}_{i}}, \tm_{w\cdot\lambda})=0$. We denote the set of Weyl elements for which $H^\bullet(\rS^{\mr{M}_{i}}, \tm_{w\cdot\lambda})\neq0$ by $\overline{\W}^{i}(\lambda)$.
\end{lema} 

Now, if $\mr{B} \subset \mr{GL}_2$ is the usual Borel subgroup and $\mr{T} \subset \mr{B}$ is the subgroup of diagonal matrices, one can consider the exact sequence in cohomology
\begin{equation}
H^0(\rS^\mr{T}, \widetilde{H^0(\mathfrak{n}, \mathcal{V}_{a, n})}) \rightarrow H_c^1(\widetilde{\rS}^{\mr{GL}_2}, \widetilde{\mathcal{V}}_{a, n}) \rightarrow H^1(\widetilde{\rS}^{\mr{GL}_2}, \widetilde{\mathcal{V}}_{a, n}) \rightarrow H^0(\rS^\mr{T}, \widetilde{H^1(\mathfrak{n}, \mathcal{V}_{a, n})}) \nonumber
\end{equation}
where $\mathfrak{n}$ is the Lie algebra of the unipotent radical $\mr{N}$ of $\mr{B}$. By using an argument similar to the one presented in Lemma~\ref{parity0}, we get
\begin{eqnarray}\label{eq:CohomGL_2}
H_c^1(\widetilde{\rS}^{\mr{GL}_2}, \widetilde{\mathcal{V}}_{a, n}) &=& H_!^1(\widetilde{\rS}^{\mr{GL}_2}, \widetilde{\mathcal{V}}_{a, n}) \quad \mbox{ if } \quad \frac{a}{2} \not\equiv \frac{n}{2} \quad \mbox{ mod } 2, \nonumber \\
H^1(\widetilde{\rS}^{\mr{GL}_2}, \widetilde{\mathcal{V}}_{a, n}) &=& H_!^1(\widetilde{\rS}^{\mr{GL}_2}, \widetilde{\mathcal{V}}_{a, n}) \quad \mbox{ if } \quad \frac{a}{2} \equiv \frac{n}{2} \quad \mbox{ mod } 2.
\end{eqnarray}

In the following subsections we make note of the cohomology groups associated to the maximal parabolic subgroups $\rP_1$ and $\rP_2$ which will be used in the computations involved to determine the boundary cohomology in the next section.

\subsubsection{Cohomology of the face $\partial_1$}
In this case, the Levi $\mr{M}_1$ is isomorphic to $\mr{GL}_2$ and therefore $H^{q}(\rS^{\mr{M}_1}, \tm_{w\cdot\lambda})=0$ for every $q\geq 2$ (see the example 2.1.3 in Subsection 2.1.2 of \cite{Harder2018} for the particular case of $\mr{GL}_2$ or Theorem 11.4.4 in \cite{BoSe73} for a more general statement). The set of Weyl representatives is given by $\W^{\rP_1}=\{e, s_1, s_1 s_2\}$ where the length of the elements are respectively $0,1,2$. By definition,
\begin{eqnarray*}
 H^{q}(\partial_1, \tm_{\lambda}) &= &\bigoplus_{w\in \W^{\rP_{1}}} H^{q-\ell(w)}(\rS^{\mr{M}_{1}}, \widetilde{\m}_{w\cdot\lambda})\\
 &=& H^{q}(\rS^{\mr{M}_{1}}, \widetilde{\m}_\lambda) \oplus H^{q-1}(\rS^{\mr{M}_{1}}, \widetilde{\m}_{{s_1\cdot\lambda}})  \oplus \, H^{q-2}(\rS^{\mr{M}_{1}}, \widetilde{\m}_{{s_1s_2}\cdot\lambda})\,. \nonumber 
 \end{eqnarray*}
 Therefore,
 \begin{eqnarray}
 H^{0}(\partial_1, \tm_{\lambda}) &= & H^{0}(\rS^{\mr{M}_{1}}, \widetilde{\m}_{\lambda})\nonumber\\
 H^{1}(\partial_1, \tm_{\lambda}) &= & H^{1}(\rS^{\mr{M}_{1}}, \widetilde{\m}_{\lambda}) \oplus H^{0}(\rS^{\mr{M}_{1}}, \widetilde{\m}_{{s_1\cdot\lambda}}) \nonumber\\
H^{2}(\partial_1, \tm_{\lambda}) &= &  H^{1}(\rS^{\mr{M}_{1}}, \widetilde{\m}_{{s_1\cdot\lambda}}) \oplus H^{0}(\rS^{\mr{M}_{1}}, \widetilde{\m}_{{s_1 s_2\cdot\lambda}})\nonumber\\
H^{3}(\partial_1, \tm_{\lambda}) &= &  H^{1}(\rS^{\mr{M}_{1}}, \widetilde{\m}_{{s_1 s_2 \cdot\lambda}}) \nonumber
\end{eqnarray}
 and for every $q \geq 4$, the cohomology groups $H^{q}(\partial_1, \tm_{\lambda})=0$.

\subsubsection{Cohomology of the face $\partial_2$}
In this case, the Levi $\mr{M}_2$ is isomorphic to $\mr{GL}_2$ and therefore $H^{q}(\rS^{\mr{M}_2}, \widetilde{\m}_{w\cdot\lambda})=0$ for every $q\geq 2$. The set of Weyl representatives is given by $W^{\mr{P}_2}=\{ e, s_2, s_2s_1\}$ where the lengths of the elements are respectively $0,1,2$. By definition,
\begin{eqnarray*}
 H^{q}(\partial_2, \tm_{\lambda}) &= &\bigoplus_{w\in \mathcal{W}^{\rP_{2}}} H^{q-\ell(w)}(S^{\mr{M}_{2}}, \widetilde{\m}_{w\cdot\lambda})\\
 &=& H^{q}(S^{\mr{M}_{2}}, \widetilde{\m}_\lambda) \oplus H^{q-1}(\rS^{\mr{M}_{2}}, \widetilde{\m}_{{s_2\cdot\lambda}})  \oplus \, H^{q-2}(S^{\mr{M}_{2}}, \widetilde{\m}_{{s_2 s_1\cdot\lambda}})\,.\nonumber 
 \end{eqnarray*}
 Therefore,
 \begin{eqnarray}
 H^{0}(\partial_2, \tm_{\lambda}) &= & H^{0}(\rS^{\mr{M}_{2}}, \widetilde{\m}_{\lambda})\nonumber\\
 H^{1}(\partial_2, \tm_{\lambda}) &= & H^{1}(\rS^{\mr{M}_{2}}, \widetilde{\m}_{\lambda}) \oplus H^{0}(\rS^{\mr{M}_{2}}, \widetilde{\m}_{{s_2\cdot\lambda}}) \nonumber\\
H^{2}(\partial_2, \tm_{\lambda}) &= &  H^{1}(\rS^{\mr{M}_{2}}, \widetilde{\m}_{{s_2\cdot\lambda}}) \oplus H^{0}(\rS^{\mr{M}_{2}}, \widetilde{\m}_{{s_2 s_1\cdot\lambda}})\nonumber\\
H^{3}(\partial_2, \tm_{\lambda}) &= &  H^{1}(\rS^{\mr{M}_{2}}, \widetilde{\m}_{{s_2 s_1 \cdot\lambda}}) \nonumber
\end{eqnarray}
 and for every $q \geq 4$, the cohomology groups $H^{q}(\partial_2, \tm_{\lambda})=0$.
 
 
 

\section{Boundary Cohomology}\label{bdsl3}

In this section we calculate the cohomology of the boundary by giving a complete description of the spectral sequence. The covering of the boundary of the Borel-Serre compactification defines a spectral sequence in cohomology.
\[
	E_1^{p, q} = \bigoplus_{prk(\mathrm{P})=(p+1)} H^q(\partial_\rP, \widetilde{\mathcal{M}}_\lambda) \Rightarrow H^{p+q}(\partial \rS, \widetilde{\mathcal{M}}_\lambda)
\]
and the nonzero terms of $E_1^{p, q}$ are for 
\begin{equation}\label{eq:pq}
(p, q) \in \left\{ (0, 0), (0, 1), (0, 2), (0, 3), (1, 0), (1, 1), (1, 2), (1, 3)\right\} \,.
\end{equation}

More precisely,  
\begin{eqnarray}\label{eq:e12}
 E_{ 1}^{0,q} & = & \bigoplus_{i=1}^{2} \mr{H}^{q}(\partial_{i}, \tm_{\lambda}) \nonumber\\
 &= &  \bigoplus_{i=1}^{2}\bigg[ \bigoplus_{w\in \W^{\rP_{i}}} \mr{H}^{q-\ell(w)}(\mr{S}^{\mr{M}_{i}}, \tm_{w\cdot\lambda})\bigg]  \,,\\
 E_{ 1}^{1,q} & = & \mr{H}^{q}(\partial_{0}, \tm_{\lambda}) \nonumber\\
& = & \bigoplus_{w \in \W^{\rP_{0}}: \ell(w)=q} \mr{H}^{0}(\mr{S}^{\mr{M}_{0}}, \widetilde{\m}_{w\cdot\lambda})\,.\nonumber
 \end{eqnarray}
 
 Since $\mr{SL}_{3}$  is of rank two, the spectral sequence has only two columns namely $E_{1}^{0,q}, E_{1}^{1,q}$ and to study the boundary cohomology, the task reduces to analyze the following morphisms

\begin{equation}\label{eq:spec3} 
E_{1}^{0,q} \xrightarrow{d_{1}^{0,q}} E_{1}^{1,q} 
\end{equation}

 where  $d_{1}^{0,q}$ is the differential map and the higher differentials vanish. One has 
 \begin{eqnarray*} 
 E_2^{0, q} := Ker(d_{1}^{0,q}) \quad \mr{and} \quad E_2^{1, q} := Coker(d_{1}^{0,q})\,.
 \end{eqnarray*}

In addition, due to be in rank $2$ situation, the spectral sequence degenerates in degree $2$. Therefore, we can use the fact that
\begin{equation}\label{eq:hks}
	H^k(\partial \rS, \tm_\lambda) = \bigoplus_{p+q=k} E_2^{p, q}\,.
\end{equation}

In other words, let us now consider the short exact sequence 

\begin{eqnarray}\label{eq:ses}
0 \longrightarrow E_2^{1, q-1} \longrightarrow H^q(\partial \rS, \tm_\lambda) \longrightarrow E_2^{0, q} \longrightarrow 0 \quad.
\end{eqnarray}

From now on, we will denote by $r_1 : \mr{H}^{\bullet}(\partial_{1}, \tm_{\lambda}) \rightarrow \mr{H}^{\bullet}(\partial_{0}, \tm_{\lambda})$ and  $r_2 : \mr{H}^{\bullet}(\partial_{2}, \tm_{\lambda}) \rightarrow \mr{H}^{\bullet}(\partial_{0}, \tm_{\lambda})$ the natural restriction morphisms.

\subsection{Case 1\,: $m_1 = 0$ and $m_2 = 0$ (trivial coefficient system)}
Following Lemma~\ref{parity0} and Lemma~\ref{parity1} from Section~\ref{parity}, we get 
\begin{eqnarray*}
\overline{\W}^{1}(\lambda) = \{e\} \,, \quad \overline{\W}^{2}(\lambda)=\{ e\} \,\quad \mr{and} \quad \overline{\W}^{0}(\lambda) = \{ e , s_1s_2s_1\} \,.
\end{eqnarray*}

By using~\eqr{e12} we record the values of $E_{1}^{0,q}$ and $E_{1}^{1,q}$ for the distinct values of $q$ below. Note that following~\eqr{pq} we know that for $q\geq 4$, $E_{1}^{i,q}=0$ for $i=0, 1$.

\begin{equation}\label{eq:e110q}
E_{1}^{0,q} =\left\{\begin{array}{cccc}  
&H^{0}(\rS^{\mr{M}_1}, \tm_{e\cdot\lambda})\oplus H^{0}(\rS^{\mr{M}_2}, \tm_{e\cdot\lambda})\cong\Q \oplus \Q \,, & q=0 \\ 
&\\
&0  \,, & \mr{otherwise}\\
\end{array}\qquad\,, \right .
\end{equation}
and 
\begin{equation}\label{eq:e111q}
E_{1}^{1,q} =\left\{\begin{array}{cccc}  
& H^{0}(\rS^{\mr{M}_0}, \tm_{e\cdot\lambda}) \cong \Q\,,& q=0\\
&\\
&  H^{0}(\rS^{\mr{M}_0}, \tm_{s_1 s_2 s_1\cdot\lambda}) \cong \Q\,,& q=3\\
&\\
& 0  \,,& \mr{otherwise}\\
\end{array}\qquad\,. \right .
\end{equation}

We now make a thorough analysis of~\eqr{spec3} to get the complete description of the spaces $E_{2}^{0,q}$ and $E_{2}^{1,q}$ which will give us the cohomology $\rH^{q}(\partial \bar{\rS}, \tm_{\lambda})$. We begin with $q=0$.

\subsubsection{{At the level $q=0$}} Observe that the short exact sequence~\eqr{ses} reduces  to 
\begin{eqnarray*}
0 \longrightarrow H^0(\partial \rS, \tm_\lambda) \longrightarrow E_2^{0, 0} \longrightarrow 0 \quad.
\end{eqnarray*}
To compute $E_2^{0, 0}$,  consider the differential $d_1^{0, 0}: E_1^{0, 0} \rightarrow E_1^{1, 0}$. Following~\eqr{e110q} and ~\eqr{e111q}, we have $d_1^{0, 0}: \Q \oplus \Q \longrightarrow \Q$ and we know  
that the differential $d_1^{0, 0}$ is surjective (see ~\cite{Harder87}).
Therefore 
\begin{align}\label{eq:e210}
 E_2^{0, 0} := Ker(d_1^{0, 0}) = \mathbb{Q} \quad \mr{and}  \quad E_2^{1, 0}:=Coker(d_1^{0, 0}) = 0  .
\end{align} 

Hence, we get $$ H^0(\partial \rS, \tm_\lambda) =\Q\,.$$

\subsubsection{{At the level $q=1$}} Following~\eqr{e210}, in this case, our short exact sequence~\eqr{ses} reduces to 
\begin{eqnarray*}
0 \longrightarrow H^1(\partial \rS, \tm_\lambda) \longrightarrow E_2^{0, 1} \longrightarrow 0 \quad\,,
\end{eqnarray*}
and we need  to compute $E_2^{0, 1}$. Consider the differential $d_{1}^{0,1}: E_1^{0, 1} \longrightarrow E_1^{1, 1}$ and following~\eqr{e110q} and~\eqr{e111q}, we observe that $d_{1}^{0,1}$ is a map between zero  spaces. Therefore, we obtain 
\begin{eqnarray*}\label{eq:e201}
E_2^{0, 1} = 0  \qquad \mr{and} \qquad E_2^{1, 1} = 0\,.
\end{eqnarray*}
As a result, we get $$ H^1(\partial \rS, \tm_\lambda) =0\,.$$

\subsubsection{{At the level $q=2$}} Following the similar process as in level $q=1$, we get 
\begin{eqnarray}\label{eq:e202}
E_2^{0, 2} = 0\qquad \mr{and} \qquad E_2^{1, 2} = 0\,.
\end{eqnarray}
This results into $$ H^2(\partial \rS, \tm_\lambda) =0\,.$$

\subsubsection{{At the level $q=3$}} Following~\eqr{e202}, in this case, the short exact sequence~\eqr{ses} reduces to 
\begin{eqnarray*}
0 \longrightarrow H^3(\partial \rS, \tm_\lambda) \longrightarrow E_2^{0, 3} \longrightarrow 0 \quad\,,
\end{eqnarray*}
and we need  to compute $E_2^{0, 3}$. Consider the differential $d_{1}^{0,3}: E_1^{0, 3} \longrightarrow E_1^{1, 3}$ and following~\eqr{e110q} and~\eqr{e111q}, we have $d_{1}^{0,3}: 0 \longrightarrow \Q$.
Therefore, \begin{eqnarray}\label{eq:e213}
E_2^{0, 3} = 0\qquad \mr{and} \qquad E_2^{1, 3} = \Q\,.
\end{eqnarray}
This gives us $$ H^3(\partial \rS, \tm_\lambda) =0\,.$$

\subsubsection{{At the level $q=4$}} Following~\eqr{e213}, in this case, the short exact sequence~\eqr{ses} reduces to 
\begin{eqnarray*}
0 \longrightarrow \Q \longrightarrow H^4(\partial \rS, \tm_\lambda) \longrightarrow E_2^{0, 4} \longrightarrow 0 \quad\,,
\end{eqnarray*}
and we need  to compute $E_2^{0, 4}$. Consider the differential $d_{1}^{0,4}: E_1^{0, 4} \longrightarrow E_1^{1, 4}$ and following~\eqr{e110q} and~\eqr{e111q}, we have $d_{1}^{0,4}: 0 \longrightarrow 0$. Therefore, \begin{eqnarray*}\label{e214}
E_2^{0, 4} = 0\qquad \mr{and} \qquad E_2^{1, 4} = 0\,,
\end{eqnarray*}
and we get $$ H^4(\partial \rS, \tm_\lambda) =\Q\,.$$
We can summarize the above discussion as follows :
\begin{equation*}
H^q(\partial \rS, \tm_\lambda) =\left\{\begin{array}{cccc}  
&\mathbb{Q} \quad \mr{for}\quad  q=0,4 \\ 
&\\
&0  \qquad \mr{otherwise}\\
\end{array}\qquad\,. \right .
\end{equation*}

\subsection{Case 2\, : $m_1 = 0$, $m_2 \neq 0$, $m_2$ even}\label{case2} Following the parity conditions established in Section~\ref{parity}, we find that 
\begin{eqnarray*}
\overline{\W}^{1}(\lambda) = \{e\} \,, \quad \overline{\W}^{2}(\lambda)=\{ e, s_2 s_1\} \,\quad \mr{and} \quad \overline{\W}^{0}(\lambda) = \{ e , s_1s_2s_1\} \,.
\end{eqnarray*}

Following~\eqr{e12} we write 
\begin{equation*}\label{eq:e10q}
E_{1}^{0,q} =\left\{\begin{array}{cccc}  
&H^{0}(\rS^{\mr{M}_2}, \tm_{e\cdot\lambda})\,,  &  q=0 \\ 
&\\
& H^{1}(\rS^{\mr{M}_1}, \tm_{e\cdot\lambda})  \,, & q=1 \\
&\\
& H^{1}(\rS^{\mr{M}_2}, \tm_{s_2s_1 \cdot \lambda})\,, & q=3\\
&\\
& 0 \,, & \mr{otherwise}
\end{array}\qquad\,, \right .
\end{equation*}
and 
\begin{equation}\label{eq:e211q}
E_{1}^{1,q} =\left\{\begin{array}{cccc}  
& H^{0}(\rS^{\mr{M}_0}, \tm_{e\cdot\lambda}) \cong \Q\,,& q=0\\
&\\
&  H^{0}(\rS^{\mr{M}_0}, \tm_{s_1 s_2 s_1\cdot\lambda}) \cong \Q\,,& q=3\\
&\\
& 0  \,,& \mr{otherwise}\\
\end{array}\qquad\,. \right .
\end{equation}
\subsubsection{{At the level $q=0$}} In this case,  the short exact sequence~\eqr{ses} is
\begin{eqnarray*}
0 \longrightarrow H^0(\partial \rS, \tm_\lambda) \longrightarrow E_2^{0, 0} \longrightarrow 0 \quad\,.
\end{eqnarray*}
Consider the differential $d_{1}^{0,0}: E_1^{0, 0} \longrightarrow E_1^{1, 0}$  which is an isomorphism $$d_{1}^{0,0}: H^{0}(\rS^{\mr{M}_2}, \tm_{e\cdot\lambda}) \longrightarrow H^{0}(\rS^{\mr{M}_0}, \tm_{e\cdot\lambda})\,.$$ Therefore, we obtain 
\begin{eqnarray*}\label{eq:e2200}
E_2^{0, 0} = 0  \qquad \mr{and} \qquad E_2^{1, 0} = 0\,.
\end{eqnarray*}
As a result, we get $$ H^0(\partial \rS, \tm_\lambda) =0\,.$$

\subsubsection{{At the level $q=1$}} In this case,  the short exact sequence~\eqr{ses} becomes
\begin{eqnarray*}
0 \longrightarrow H^1(\partial \rS, \tm_\lambda) \longrightarrow E_2^{0, 1} \longrightarrow 0 \quad\,.
\end{eqnarray*}
Consider the differential $d_{1}^{0,1}: E_1^{0, 1} \longrightarrow E_1^{1, 1}$  which, from~\eqr{CohomGL_2}, is simply a zero morphism $$d_{1}^{0,1}: H_!^{1}(\rS^{\mr{M}_1}, \tm_{e\cdot\lambda}) \longrightarrow 0\,.$$ Therefore, we obtain 
\begin{eqnarray*}\label{eq:e2201}
E_2^{0, 1} = H_!^{1}(\rS^{\mr{M}_1}, \tm_{e\cdot\lambda})  \qquad \mr{and} \qquad E_2^{1, 1} = 0\,.
\end{eqnarray*}
As a result, we get $$ H^1(\partial \rS, \tm_\lambda) =  H_!^{1}(\rS^{M_1}, \tm_{e\cdot\lambda})\,.$$

\subsubsection{{At the level $q=2$}} The short exact sequence becomes 
\begin{eqnarray*}
0 \longrightarrow H^2(\partial \rS, \tm_\lambda) \longrightarrow E_2^{0, 2} \longrightarrow 0 \quad \,, 
\end{eqnarray*}
and following the differential $d_{1}^{0,2}: E_1^{0, 2} \longrightarrow E_1^{1, 2}$  which is again simply the zero morphism $$d_{1}^{0,2}: 0 \longrightarrow 0\,,$$ gives us 
\begin{eqnarray}
E_2^{0, 2} = 0 \qquad \mr{and} \qquad E_2^{1, 2} = 0\,.\nonumber
\end{eqnarray}
Hence, $$ H^2(\partial \rS, \tm_\lambda) = 0 \,.$$

\subsubsection{{At the level $q=3$}}  The short exact sequence~\eqr{ses} reduces to 
\begin{eqnarray*}
0 \longrightarrow H^3(\partial \rS, \tm_\lambda) \longrightarrow E_2^{0, 3} \longrightarrow 0 \quad\,,
\end{eqnarray*}
and the differential $d_{1}^{0,3}: E_1^{0, 3} \longrightarrow E_1^{1, 3}$ is an epimorphism 
$$ d_{1}^{0,3}:H^1(\rS^{\mr{M}_2}, \tm_{s_2s_1\cdot\lambda}) \longrightarrow H^0(\rS^{\mr{M}_0}, \tm_{s_1s_2s_1.\lambda})\,,$$
 
Therefore \begin{align*}
E_2^{1, 3} = 0 \quad \mr{and} \quad E_2^{0, 3} = H_!^1(\rS^{\mr{M}_2}, \tm_{s_2s_1.\lambda}).
\end{align*}

Since $E_2^{1, 3} = 0$ and $E_2^{0, 4} =0$, we realize that $H^q(\partial \rS, \tm_\lambda) = 0$ for every $q\geq 4$. \\

We summarize the discussion of this case as follows 
\begin{equation*}
H^q(\partial \rS, \tm_\lambda) =\left\{\begin{array}{cccc}  
& H_{!}^{1}(\rS^{\mr{M}_1}, \tm_{e.\lambda}) \,, & q=1 \\ 
&\\
& H_{!}^{1}(\rS^{\mr{M}_2}, \tm_{s_2s_1\cdot\lambda}) \,, & q=3 \\
&\\
&0  \,, & \mr{otherwise}\\
\end{array}\qquad\,. \right .
\end{equation*}

\subsection{Case 3 : $m_2 = 0$, $m_1 \neq 0$, $m_1$ even}\label{case3} Following the parity conditions established in Section~\ref{parity}, we find that 
\begin{eqnarray*}
\overline{\W}^{1}(\lambda) = \{e, s_1 s_2\} \,, \quad \overline{\W}^{2}(\lambda)=\{ e \} \,\quad \mr{and} \quad \overline{\W}^{0}(\lambda) = \{ e , s_1s_2s_1\} \,.
\end{eqnarray*}

Following~\eqr{e12}, 
\begin{equation*}\label{eq:e310q}
E_{1}^{0,q} =\left\{\begin{array}{cccc}  
&H^{0}(\rS^{\mr{M}_1}, \tm_{e\cdot\lambda})\,,  &  q=0 \\ 
&\\
& H^{1}(\rS^{\mr{M}_2}, \tm_{e\cdot\lambda})  \,, & q=1 \\
&\\
& H^{1}(\rS^{\mr{M}_1}, \tm_{s_1s_2 \cdot \lambda})\,, & q=3\\
&\\
& 0 \,, & \mr{otherwise}
\end{array}\qquad\,, \right .
\end{equation*}
and 
the spaces $E_{1}^{1,q}$ in this case are exactly same as described in the above two cases expressed by~\eqr{e211q}. Following similar steps  taken  in Subsection~\ref{case2}, we obtain the following 
\begin{equation*}
H^q(\partial \rS, \tm_\lambda) =\left\{\begin{array}{cccc}  
& H_{!}^{1}(\rS^{\mr{M}_2}, \tm_{e.\lambda}) \,, & q=1 \\ 
&\\
& H_{!}^{1}(\rS^{\mr{M}_1}, \tm_{s_1s_2\cdot\lambda}) \,, & q=3 \\
&\\
&0  \,, & \mr{otherwise}\\
\end{array}\qquad\,. \right .
\end{equation*}

\subsection{Case 4 : $m_1 \neq 0$, $m_1$ even  and $m_2 \neq 0$,  $m_2$ even}\label{case4} Following the parity conditions established in Section~\ref{parity}, we find that 
\begin{eqnarray*}
\overline{\W}^{1}(\lambda) = \{e, s_1 s_2\} \,, \quad \overline{\W}^{2}(\lambda)=\{ e, s_2 s_1\} \,\quad \mr{and} \quad \overline{\W}^{0}(\lambda) = \{ e , s_1s_2s_1\} \,.
\end{eqnarray*}

Following~\eqr{e12}, 
\begin{equation*}\label{eq:e410q}
E_{1}^{0,q} =\left\{\begin{array}{cccc}  
& H^{1}(\rS^{\mr{M}_2}, \tm_{e\cdot\lambda})\oplus H^{1}(\rS^{\mr{M}_2}, \tm_{e\cdot\lambda})  \,, & q=1 \\
&\\
& H^{1}(\rS^{\mr{M}_1}, \tm_{s_1s_2 \cdot \lambda}) \oplus H^{1}(\rS^{\mr{M}_2}, \tm_{s_2s_1 \cdot \lambda})\,, & q=3\\
&\\
& 0 \,, & \mr{otherwise}
\end{array}\qquad\,, \right .
\end{equation*}
and the spaces $E_{1}^{1,q}$ are described by~\eqr{e211q}. Combining the process performed for the previous two cases in Subsections~\ref{case2} and ~\ref{case3}, we get the following result

\begin{equation*}
H^q(\partial \rS, \m_\lambda) =\left\{\begin{array}{cccc}  
& \Q \oplus H_{!}^{1}(\rS^{\mr{M}_1}, \tm_{e\cdot\lambda})\oplus H_{!}^{1}(\rS^{\mr{M}_2}, \tm_{e\cdot\lambda}) \,, & q=1 \\ 
&\\
&   H_{!}^1(\rS^{\mr{M}_1}, \tm_{s_1s_2\cdot\lambda}) \oplus H_{!}^1(\rS^{\mr{M}_2}, \tm_{s_2s_1.\lambda}) \oplus \Q \,, & q=3 \\
&\\
&0  \,, & \mr{otherwise}\\
\end{array}\qquad\,. \right .
\end{equation*}

\subsection{Case 5 : $m_1 \neq 0$, $m_1$ even, $m_2$ odd}\label{case5} 
  
Following the parity conditions established in Section~\ref{parity} and~\eqr{e12}, we find that 
\begin{eqnarray*}
\overline{\W}^{1}(\lambda) = \{s_1, s_1 s_2\} \,, \quad \overline{\W}^{2}(\lambda)=\{ e, s_2\} \,\quad \mr{and} \quad \overline{\W}^{0}(\lambda) = \{ s_1 , s_1s_2\} \,,
\end{eqnarray*}
and
\begin{equation*}\label{eq:e510q}
E_{1}^{0,q} =\left\{\begin{array}{cccc}  
&H^{1}(\rS^{\mr{M}_2}, \tm_{e\cdot\lambda})\,,  &  q=1 \\ 
&\\
& H^{1}(\rS^{\mr{M}_1}, \tm_{s_1 \cdot \lambda}) \oplus H^{1}(\rS^{\mr{M}_2}, \tm_{s_2 \cdot \lambda}) \,, & q=2\\
&\\
& H^{1}(\rS^{\mr{M}_1}, \tm_{s_1s_2 \cdot \lambda})\,, & q=3\\
&\\
& 0 \,, & \mr{otherwise}
\end{array}\qquad\,, \right .
\end{equation*}
and 
\begin{equation}\label{eq:e511q}
E_{1}^{1,q} =\left\{\begin{array}{cccc}  
& H^{0}(\rS^{\mr{M}_0}, \tm_{s_1\cdot\lambda}) \cong \Q\,, & q=1\\
&\\
&H^{0}(\rS^{\mr{M}_0}, \tm_{s_1 s_2 \cdot\lambda}) \cong \Q \,, & q=2\\ 
&\\
& 0  \,,& \mr{otherwise}\\
\end{array}\qquad\,. \right .
\end{equation}
Following the similar computations we get all the spaces $E_{2}^{p,q}$ for $p=0,1$ as follows
\begin{equation*}\label{eq:e520}
E_{2}^{0,q} =\left\{\begin{array}{cccc}  
&H^{1}_{!}(\rS^{\mr{M}_2}, \tm_{e\cdot\lambda})\,,  &  q=1 \\ 
&\\
& H^{1}_{!}(\rS^{\mr{M}_1}, \tm_{s_1 \cdot \lambda}) \oplus H^{1}_{!}(\rS^{\mr{M}_2}, \tm_{s_2 \cdot \lambda}) \,, & q=2\\
&\\
& H^{1}_{!}(\rS^{\mr{M}_1}, \tm_{s_1s_2 \cdot \lambda})\,, & q=3\\
&\\
& 0 \,, & \mr{otherwise}
\end{array}\qquad\,, \right .
\end{equation*}
and 
\begin{equation*}\label{eq:e521}
E_{2}^{1,q} =0\,, \quad \forall q\,. 
\end{equation*}
Following~\eqr{hks}, we obtain 
\begin{eqnarray*}
H^q(\partial \rS, \tm_\lambda) = E_{2}^{0,q}\,, \qquad \forall  \, q\,.
  \end{eqnarray*}

\subsection{Case 6 : $m_1 = 0$, $m_2$ odd}\label{case6} Following the parity conditions established in Section~\ref{parity} and~\eqr{e12}, we find that 
\begin{eqnarray*}
\overline{\W}^{1}(\lambda) = \{s_1, s_1 s_2\} \,, \quad \overline{\W}^{2}(\lambda)=\{s_2\} \,\quad \mr{and} \quad \overline{\W}^{0}(\lambda) = \{ s_1 , s_1s_2\} \,,
\end{eqnarray*}
and
\begin{equation*}\label{eq:e10qq}
E_{1}^{0,q} =\left\{\begin{array}{cccc}  
& H^{1}(\rS^{\mr{M}_1}, \tm_{s_1 \cdot \lambda})\oplus H^{0}(\rS^{\mr{M}_1}, \tm_{s_1 s_2 \cdot \lambda})  \oplus H^{1}(\rS^{\mr{M}_2}, \tm_{s_2 \cdot \lambda}) \,, & q=2\\
&\\
& 0 \,, & \mr{otherwise}
\end{array}\qquad\,, \right .
\end{equation*}

and the spaces $E_{1}^{1,q}$ are described by~\eqr{e511q}. Following the similar computations we get all the spaces $E_{2}^{p,q}$ for $p=0,1$ as follows
\begin{equation*}\label{eq:e5520}
E_{2}^{0,q} =\left\{\begin{array}{cccc}  
& H_{!}^1(\rS^{\mr{M}_1}, \tm_{s_1\cdot\lambda}) \oplus W \oplus H_{!}^1(\rS^{\mr{M}_2}, \tm_{s_2\cdot\lambda})\,, & q=2\\
&\\ 
& 0 \,, & \mr{otherwise}
\end{array}\qquad\,, \right .
\end{equation*}
where $W$ is the one dimensional space
\begin{equation*}\label{eq:W6}
W = \left\{(\xi, \nu) \in  H^0(\rS^{\mr{M}_1}, \tm_{s_1s_2\cdot\lambda}) \oplus H_{Eis}^1(\rS^{\mr{M}_2}, \tm_{s_2\cdot\lambda}) \mid r_1(\xi) = r_2(\nu) \right\}
\end{equation*} 
along with $r_1$ and $r_2$ the restriction morphisms defined as follows
\begin{align}
H^{\bullet}(\rS^{\rm{M}_2}, \tm_{s_2 \cdot \lambda}) \xrightarrow{r_2} H^{\bullet}(\rS^{\rm{M}_0}, \tm_{s_1 s_2 \cdot \lambda}) \nonumber \\
H^{\bullet}(\rS^{\rm{M}_1}, \tm_{s_1 s_2 \cdot \lambda}) \xrightarrow{r_1} H^{\bullet}(\rS^{\rm{M}_0}, \tm_{s_1 s_2 \cdot \lambda})  \nonumber
\,.\end{align}

Both $r_1$ and $r_2$ are surjective. This fact follows directly by applying Kostant's formula to the Levi quotient of each of the maximal parabolic subgroups. Then, the target spaces of $r_1$ and $r_2$ are just the boundary and the Eisenstein cohomology of GL$_2$, respectively. From the above properties of $r_1$ and $r_2$, we conclude that $W$ is isomorphic to $H^{0}(\rS^{\rm{M}_0}, \tm_{s_1 s_2 \cdot \lambda})$, which is a $1$-dimensional space.

However,
\begin{equation*}\label{eq:e621}
E_{2}^{1,q} =\left\{\begin{array}{cccc}
&H^{0}(\rS^{\mr{M}_2}, \tm_{s_1\cdot\lambda}) \cong \Q\,,  &  q=1 \\ 
&\\
&0\,, & \mr{otherwise}
\end{array}\qquad\,, \right .
\end{equation*}
Now, following~\eqr{hks}, we obtain 
\begin{eqnarray*}
H^q(\partial \rS, \tm_\lambda) =  \left\{\begin{array}{cccc}
& H_{!}^1(\rS^{\mr{M}_1}, \tm_{s_1\cdot\lambda}) \oplus H_{!}^1(\rS^{\mr{M}_2}, \tm_{s_2\cdot\lambda}) \oplus \Q \oplus \Q \,, & q=2 \\
&\\
& 0 \,,& \mr{otherwise}
\end{array}\qquad\,, \right .
\end{eqnarray*}

\subsection{Case 7:  $m_1$ odd, $m_2 =0$} Following the parity conditions established in Section~\ref{parity} we find that 
\begin{eqnarray*}
\overline{\W}^{1}(\lambda) = \{s_1 \} \,, \quad \overline{\W}^{2}(\lambda)=\{s_2, s_2 s_1\} \,\quad \mr{and} \quad \overline{\W}^{0}(\lambda) = \{ s_2 , s_2 s_1\} \,.
\end{eqnarray*}
Observe that this is exactly the reflection of case 6 described in Subsection~\ref{case6}. The roles of parabolics $\rP_1$ and $\rP_2$ will be interchanged. Hence, following the similar arguments we will obtain
\begin{eqnarray*}
H^q(\partial \rS, \tm_\lambda) =  \left\{\begin{array}{cccc}
& H_{!}^1(\rS^{\mr{M}_2}, \tm_{s_2\cdot\lambda}) \oplus H_{!}^1(\rS^{\mr{M}_1}, \tm_{s_1\cdot\lambda}) \oplus \Q \oplus \Q \,, & q=2 \\
& \\
& 0 \,,& \mr{otherwise}
\end{array}\qquad\,. \right .
\end{eqnarray*}
 
\subsection{Case 8 :  $m_1$ odd, $m_2 \neq 0$, $m_2$ even} Following the parity conditions established in Section~\ref{parity} we find that 
\begin{eqnarray*}
\overline{\W}^{1}(\lambda) = \{e, s_1 \} \,, \quad \overline{\W}^{2}(\lambda)=\{ s_2, s_2 s_1\} \,\quad \mr{and} \quad \overline{\W}^{0}(\lambda) = \{ s_2 , s_2 s_1\} \,.
\end{eqnarray*}
Observe that this is exactly the reflection of case 5 described in Subsection~\ref{case5}. The roles of parabolics $\rP_1$ and $\rP_2$ will be interchanged. Hence,  we will obtain
\begin{equation*}\label{eq:e55520}
H^q(\partial \rS, \tm_\lambda) =\left\{\begin{array}{cccc}  
&H^{1}_{!}(\rS^{\mr{M}_1}, \tm_{e\cdot\lambda})\,,  &  q=1 \\ 
&\\
& H^{1}_{!}(\rS^{\mr{M}_1}, \tm_{s_1 \cdot \lambda}) \oplus H^{1}_{!}(\rS^{\mr{M}_2}, \tm_{s_2 \cdot \lambda}) \,, & q=2\\
&\\
& H^{1}_{!}(\rS^{\mr{M}_2}, \tm_{s_2s_1 \cdot \lambda})\,, & q=3\\
&\\
& 0 \,, & \mr{otherwise}
\end{array}\qquad\,. \right .
\end{equation*}

\subsection{Case 9:   $m_1 $ odd, $m_2$ odd}
By checking the parity conditions for standard parabolics, following Lemmas ~\ref{parity0} and ~\ref{parity1}, we see that $\overline{\W}^{i}(\lambda) = \emptyset$ for $i=0,1,2$. This simply implies that 
$$H^{q}(\partial \rS, \tm_\lambda) = 0 \,, \qquad \forall q \,.$$




\section{Euler Characteristic}\label{Euler}

We quickly review the basics about Euler characteristic which is our important tool to obtain the information about Eisenstein cohomology discussed in the next section.  The homological Euler characteristic $\chi_h$ of a group $\Gamma$ with coefficients in a representation $\mc{V}$ is defined by 
\begin{equation}\label{eq:hec}
\chi_h(\Gamma,\mc{V})=\sum_{i=0}^{\infty}\, (-1)^{i} \, \mr{dim} \, H^{i}(\Gamma,\mc{V}). 
\end{equation}
For details on the above formula  see~\cite{Brown94, Serre71}. We recall the definition of orbifold Euler characteristic. If $\Gamma$ is torsion free, then the orbifold Euler characteristic is defined as $\chi_{orb}(\Gamma) = \chi_h(\Gamma)$. If $\Gamma$ has torsion elements and admits a finite index torsion free subgroup $\Gamma'$, then the orbifold Euler characteristic of $\Gamma$ is given by 
\begin{equation}\label{eq:orbeuler}
\chi_{orb}(\Gamma) = \frac{1}{[\Gamma : \Gamma']} \chi_h(\Gamma')\,.
\end{equation} 
One important fact is that, following Minkowski, every arithmetic group of rank greater than one contains a torsion free finite index subgroup and therefore the concept of orbifold Euler characteristic is well defined in this setting. If $\Gamma$ has torion elements then we make use of the following  formula discovered by Wall in~\cite{Wall}.
\begin{equation}\label{eq:hecT}
\chi_h(\Gamma,\mc{V} )=\sum_{(T)} \chi_{orb}(C(T)) tr(T^{-1}| \mc{V}) .
\end{equation}
 Otherwise, we use the formula described in equation~\eqr{hec}. The sum runs over all the conjugacy classes in $\Gamma$ of its torsion elements $T$, denoted by $(T)$, and $C(T)$ denotes the centralizer of $T$ in $\Gamma$. From now on, orbifold Euler characteristic $\chi_{orb}$ will be simply denoted by $\chi$. Orbifold Euler characteristic has the following properties.
 \begin{enumerate}
\item  If $\Gamma$  is finitely generated  torsion free group  then $\chi(\Gamma)$ is defined as $\chi_h(\Gamma,\Q).$
\item If $\Gamma$ is finite of order $\left| \Gamma \right|$ then $\chi(\Gamma)=\frac{1}{\left| \Gamma \right|}$.
\item Let $\Gamma$, $\Gamma_1$ and  $\Gamma_2$ be groups such that 
$ 1 \longrightarrow \Gamma_1 \longrightarrow \Gamma \longrightarrow \Gamma_2 \longrightarrow 1$ is exact then $\chi(\Gamma)= \chi(\Gamma_1) \chi(\Gamma_2)$.
 \end{enumerate}
 
 We now explain the use of the above properties by walking through the detailed computation of the Euler characteristic of $\mr{SL}_2(\Z)$ and $\mr{GL}_{2}(\Z)$ with respect to their highest weight representations, which we explain shortly. \\
 
\par We denote
\[ 
T_3 = \left( \begin{array}{rrr}
0 & 1  \\
-1 & -1 \\ 
\end{array}  \right),
T_4 = \left( \begin{array}{rrr}
0 & 1 \\
-1 & 0 \\ 
\end{array}  \right) \mbox{ and }
T_6 = \left( \begin{array}{rrr}
0 & -1  \\
1 & 1   \\ 
\end{array}  \right).
\]
Then following~\cite{Horozov2005}, we know that when $\Gamma$ is $\mr{GL}_n(\Z)$ (or $\mr{SL}_n (\Z)$ with $n$ odd) one has an expression of the form 
 \begin{equation}\label{eq:hnthor}
 \chi_{h}(\Gamma, \mc{V}) = \sum_{A} Res(f_A) \chi(C(A)) Tr(A^{-1}| \mc{V})\,,
 \end{equation}
where $f_A$ denotes the characteristic polynomial of the matrix $A$.

Now we will explain equation~\eqref{eq:hnthor} in detail. The summation is over all possible block diagonal matrices $A \in \Gamma$ satisfying the following conditions:
\begin{itemize}
\item The blocks in the diagonal belong to the set $\left\{1, -1, T_3, T_4, T_6\right\}$.
\item The blocks $T_3, T_4$ and $T_6$ appear at most once and $1, -1$ appear at most twice.
\item A change in the order of the blocks in the diagonal does not count as a different element.
\end{itemize}
So, for example, if $n > 10$, the sum is empty and $\chi_{h}(\Gamma, \mc{V}) = 0$.

In this case, one can see that every $A$ satisfying these properties has the same eigenvalues as $A^{-1}$. Even more every such $A$ is conjugate, over $\mathbb{C}$, to $A^{-1}$ and therefore $Tr(A^{-1}| \mc{V}) = Tr(A| \mc{V})$. We will use these facts in what follows.

\par For other groups, the analogous formula of~\eqr{hecT} is developed by Chiswell in \cite{Chiswell76}. Let us explain briefly the notation $Res(f)$. Let $f_1=\prod_i(x - \alpha_i)$ and $f_2=\prod_j(x-\beta_j)$ be two polynomials. Then by the resultant of $f_1$ and $f_2$, we mean $Res(f_1,f_2)=\prod_{i,j}(\alpha_i-\beta_j)$. If the characteristic polynomial $f$ is a power of an irreducible polynomial then we define $Res(f)=1$. Let $f=f_1 f_2\dots f_d$, where each  $f_i$ is a power of an irreducible polynomial over $\Q$ and they are relatively prime pairwise. Then, we define $Res(f)=\prod_{i<j} Res(f_i,f_j)$.

\subsection{Example : Euler Characteristic of $\mr{SL}_2(\Z)$ and  $\mr{GL}_2(\Z)$}

Consider the group $\overline{\Gamma}_0 = \mr{SL}_2(\Z)/ \{ \pm I_2\}$. For any subgroup $\G \in \mr{SL}_2(\Z)$ containing $-I_2$, we will denote by $\overline{\G}$ its corresponding subgroup in  $\overline{\Gamma}_0$, \ie $\overline{\G}=\G/\{ \pm I_2\}$. 

Consider the principal congruence subgroup $\overline{\G}(2)$. It is of index $6$  and torsion free. More precisely, ${}_{\overline{\G}(2)}\backslash \h $ is topologically $\mathbb{P}^{1}-\{0,1,\infty \}$. Therefore, $$\chi(\overline{\G}(2))= \chi(\mathbb{P}^{1}-\{0,1,\infty \}) = \chi(\mathbb{P}^{1}) -3 =2 -3 =-1.$$ 
Using this we immediately get $$\chi(\G(2))= \chi(\overline{\G}(2)) \chi(\{\pm I_2 \})= -1 \times \frac{1}{2}=-\frac{1}{2}.$$

Considering the following short exact sequence $$1\longrightarrow \G(2) \longrightarrow \mr{SL}_2(\Z) \longrightarrow \mr{SL}_2(\Z/2\Z)\longrightarrow 1$$

we obtain $\chi(\mr{SL}_2(\Z))= -\frac{1}{12}$ and $\chi(\overline{\Gamma}_0)= -\frac{1}{6}$. Similarly, the exact sequence $$1\longrightarrow \mr{SL}_2(\Z) \rightarrow \mr{GL}_2(\Z) \xrightarrow{det} \{ \pm I_2\} \longrightarrow 1$$ where $det : \mr{GL}_2(\Z) \longrightarrow \{ \pm I_2\}$ is simply the determinant map, gives $\chi(\mr{GL}_2(\Z))= -\frac{1}{24}$.

For any torsion free arithmetic subgroup $\Gamma \subset \mr{SL}_n(\R)$ we have the Gauss-Bonnet formula
$$\chi_h(\Gamma \backslash X) = \int_{\Gamma \backslash X} \omega_{GB}$$
where $\omega_{GB}$ is the Gauss-Bonnet-Chern differential form and $X = \mr{SL}_n(\mathbb{R}) / \mr{SO}(n, \mathbb{R})$, see~\cite{Harder71}. This differential form is zero if $n > 2$ and therefore for any torsion free congruence subgroup $\Gamma \subset \mr{SL}_n(\mathbb{Z})$, $\chi_h(\Gamma \backslash X) = 0$.
In particular, by the definition of orbifold Euler characteristic given by~\eqr{orbeuler},  this implies that $\chi(\mr{SL}_3(\mathbb{Z})) = 0$. We will make use of this fact in the calculation of the homological Euler characteristic of $\mr{SL}_3(\mathbb{Z})$. 

In the preceding analysis, all the $\chi(\G)$ have been computed with respect to the trivial coefficient system. In case of nontrivial coefficient system, the whole game of computing $\chi(\G)$ becomes slightly delicate and interesting. To deliver the taste of its complication we quickly motivate the reader by reviewing the computations of $\chi(\mr{SL}_2(\Z), V_m)$ and $\chi(\mr{GL}_2(\Z), V_{m_1, m_2})$ where $V_m$  and $V_{m_1, m_2}$ are the highest weight irreducible representations of $\mr{SL}_2$ and $\mr{GL}_2$ respectively. For notational convenience we will always denote the standard representation of $\mr{SL}_n(\Z)$ and $\mr{GL}_n(\Z)$ by $V$. In case of $\mr{SL}_2$ and $\mr{GL}_2$, all the highest weight representations are of the form $V_m := Sym^{m} V$ and  $V_{m_1, m_2}:=Sym^{m_1} V \otimes det^{m_2}$ respectively. Here $Sym^{m} V$ denotes the $m^{th}$-symmetric power of the standard representation $V$.

Let $\Phi_n$ be the $n$-th cyclotomic polynomial then we list all the characteristic polynomials of torsion elements in  $\mr{SL}_2(\Z)$ and $\mr{GL}_2(\Z)$  in the following table. 

{ \begin{center}
\scriptsize\renewcommand{\arraystretch}{2}
\begin{longtable}{|c|c|c|c|c|c|c|c|c|}
\hline
S.No.  &  Polynomial   &  Expanded form & In $\mr{SL}_2(\Z)$ & S.No.  &  Polynomial   &  Expanded form &   In $\mr{SL}_2(\Z)$   \\
\hline
\hline
1 &  $\Phi_{_1}^{^2}$  & $(x-1)^2$  & Yes & 2 & $\Phi_{_1} \Phi_{_2}$  &  $(x-1)  (x+1)$  &  No \\  
\hline
3 & $\Phi_{_2}^{2}$  &  $(x+1)^2$  & Yes & 4 & $\Phi_{_3}$  &  $x^2 +x +1$ & Yes  \\
\hline
5 & $\Phi_{_4}$ & $x^2+1$  & Yes & 6 & $\Phi_{_6}$  & $x^2 -x +1 $&  Yes\\
\hline
\caption{Torsion elements in $\mr{GL}_2(\Z)$.}\label{torsiojngl2}
\end{longtable}
\end{center}
}

Following equation~\eqr{hecT}, we compute the traces of all the torsion elements $T$ in $\mr{SL}_2(\Z)$ and $\mr{GL}_2(\Z)$ with respect to the highest weight representations $V_{m}$ and $V_{m_1, m_2}$ for $\mr{SL}_2$  and $\mr{GL}_2$ respectively. 

For any torsion element $T \in \mr{SL}_2(\Z) $, we define 
\begin{equation*}\label{eq:hnt} H_{m}(T) := Tr(T^{-1}| V_{m})  = Tr(T^{-1} | Sym^{m} V)  =  \sum_{a+b =m}  \lambda_{1}^{a} \lambda_{2}^{b}\,. \end{equation*}
where $\lambda_1 $ and $\lambda_2$ are the two eigenvalues of $T$. From now on we simply denote the representative  of $n$ torsion element $T$  by its characteristic polynomial $\Phi_{n}$.  Therefore, 
{\begin{center}
\scriptsize\renewcommand{\arraystretch}{2}
\begin{longtable}{|c|c|c|c|c|c|c|}
\hline
Case & $T$ &  $\Phi_n$ &$C(T)$   &  $\chi(C(T))$ & $H_{m}(T)$    \\
\hline
\hline
A & $I_2$ & $\Phi_1^{2}$& $\mr{SL}_2(\Z)$ &  $-\frac{1}{12}$ & $ m+1$   \\
\hline
B &$- I_2$ &$\Phi_2^{2}$  &$\mr{SL}_2(\Z)$ &  $-\frac{1}{12}$ & $(-1)^{m}(m+1)$    \\
\hline
C &${\tmt{0}{1}{-1}{-1}}^{\pm} $ & $\Phi_3$  &$C_6$    & $\frac{1}{6}$  &$ (1, -1, 0)$\footnote{$(1,-1,0)$ signifies $H_{3k}(T)=1$, $H_{3k+1}(T)=-1$ and $H_{3k+2}(T)=0$.}   \\
\hline
D &${\tmt{0}{1}{-1}{0}}^{\pm}$ &   $\Phi_4$ &$C_4$    & $\frac{1}{4}$ & $(1, 0, -1, 0 )$     \\
\hline
E &${\tmt{0}{-1}{1}{1}}^{\pm}$ & $\Phi_6$   &$C_6$    & $\frac{1}{6}$ & $ (1, 1, 0 , -1, -1, 0)$   \\
\hline
\caption{Traces of torsion elements of $\mr{SL}_2(\Z)$.}\label{tracetorsiojnsl2}
\end{longtable}
\end{center}
}

Now following equations~\eqr{hec} and~\eqr{hecT} 
\begin{eqnarray}\label{eq:chisl2}
 \chi_{h}(\mr{SL}_2(\Z), V_m ) & = & -\frac{1}{12} H_{m}(\Phi_1^{2})  - \frac{1}{12} H_{m}(\Phi_2^{2})  + \frac{2}{6} H_{m}(\Phi_3)  + \frac{2}{4} H_{m}(\Phi_4)  + \frac{2}{6} H_{m}(\Phi_6)\,.
  \end{eqnarray}
We obtain the values of $\chi_{h}(\mr{SL}_2(\Z), V_m) $ by computing each factor of the above equation~\eqr{chisl2} up to modulo 12. All these values can be found in the last column of the Table~\ref{eulersl2gl2} below.

Similarly, let us discuss the $\chi_{h}(\mr{GL}_2(\Z), V_{m_1, m_2})$. One has the following table,
{ \begin{center}
\scriptsize\renewcommand{\arraystretch}{2.3}
\begin{longtable}{|c|c|c|c|c|c|c|c|c|}
\hline
Case & $T$ & $\Phi_n$  &$C(T)$   &  $\chi(C(T))$ & $Res(f)$    & $ H_{m_1,m_2}(T) $   \\
\hline
\hline
A & $I_2$ &$\Phi_1^2$ & $\mr{GL}_2(\Z)$ &  $-\frac{1}{24}$& $1$  & $m_1 +1$\\
\hline
B &$- I_2$ & $\Phi_2^2$& $\mr{GL}_2(\Z)$ &  $-\frac{1}{24}$ & $1$ & $ (-1)^{m_1} (m_1 +1)$ \\
\hline
C &$\tmt{1}{0}{0}{-1}$ & $\Phi_1 \Phi_2$&$C_2 \times C_2$ & $\frac{1}{4}$& 2 & $(-1)^{m_2}$ if $m_1$ is even; $0$ otherwise \\
\hline
D &$\tmt{0}{1}{-1}{-1}$ &  $\Phi_3$ &$C_6$    & $\frac{1}{6}$ & $1$ & 
$(1,-1,0)$\footnote{$(1,-1,0)$ signifies that for $m_1=3k$ or $3k+1$ or $3k+2$, we have $H_{3k,m_2}(T)=1$, $H_{3k+1,m_2}(T)=-1$ and
$H_{3k+2,m_2}(T)=0$, independently of $m_2$. Similarly, $m_1$ is taken modulo $4$ in Case $E$, and modulo $6$ in Case $F$.}\\
\hline
E &$\tmt{0}{1}{-1}{0}$ &  $\Phi_4$ & $C_4$    & $\frac{1}{4}$ & 1 & $(1,0,-1,0)$\\
\hline
F &$\tmt{0}{-1}{1}{1}$ & $\Phi_6$  & $C_6$    & $\frac{1}{6}$ & 1 & $(1,1,0,-1,-1,0)$\\
\hline
\caption{Traces of torsion elements of $\mr{GL}_2(\Z)$.}\label{tracestorsiojngl2}
\end{longtable}
\end{center}
}

Now following equations~\eqr{hec} and~\eqr{hnthor}
{
\begin{eqnarray}\label{eq:chigl2}
 \chi_{h}(\mr{GL}_2(\Z), V_{m_1, m_2} ) & = & -\frac{1}{24} H_{m_1, m_2}(\Phi_1^{2})  - \frac{1}{24} H_{m_1, m_2}(\Phi_2^{2})  - \frac{2}{4} H_{m_1, m_2}(\Phi_1 \Phi_2)  \\ \nonumber
 & + & \frac{1}{6} H_{m_1, m_2}(\Phi_3)  + \frac{1}{4} H_{m_1, m_2}(\Phi_4)  + \frac{1}{6} H_{m_1, m_2}(\Phi_6)\,.
  \end{eqnarray}
}

Same as in the case of $\mr{SL}_2(\Z)$, we obtain the values of $\chi_{h}(\mr{GL}_2(\Z), V_{m_1, m_2}) $ by computing each factor of the above equation~\eqr{chigl2} up to $m_1$ modulo 12 and $m_2$ modulo 2. All these values are encoded in the second and third column of the Table~\ref{eulersl2gl2} below. Note that in what follows $V_m$ will denote $V_{m,0}$ when it is considered as a representation of $\mathrm{GL}_2$.

It is well known that 
\begin{equation}\label{eq:blah}
S_{m+2}=H^1_{cusp}(\mr{GL}_2(\Z),V_m \otimes \mathbb{C})\subset H^1_!(\mr{GL}_2(\Z),V_m\otimes \mathbb{C})\subset H^1(\mr{GL}_2(\Z),V_m\otimes \mathbb{C}).
\end{equation}
One can show that in fact these inclusions are isomorphisms because $H^1(\mr{GL}_2(\Z),\mathbb{C}) = 0$, and for $m > 0$ we have $H^0(\mr{GL}_2(\Z),V_m) = H^2(\mr{GL}_2(\Z),V_m) = 0$ and therefore
\begin{equation*}
\dim H^1(\mr{GL}_2(\Z),V_m)= -\chi_h(\mr{GL}_2(\Z),V_m)= \dim S_{m+2}.
\end{equation*}
Hence, we may conclude that for all $m$ $$H^1_{cusp}(\mr{GL}_2(\Z),V_m \otimes \mathbb{C})= H^1(\mr{GL}_2(\Z),V_m \otimes \mathbb{C}) =  H^1_!(\mr{GL}_2(\Z),V_m) \otimes \mathbb{C}.$$ 

\begin{rmk}
Note that if we do not want to get into the transcendental aspects of the theory of cusp forms (Eichler-Shimura isomorphism) then we could get the dimension of $S_{m+2}$ by using the information given in Section 2.1.3 from Chapter 2 of~\cite{Harder2018}. 
\end{rmk}

\newpage

{\begin{center}
\scriptsize\renewcommand{\arraystretch}{2}
\begin{longtable}{|c|c|c|c|c|c|c|c|c|}
\hline
$m=12\ell +k $ & $ \chi_{h}(\mr{GL}_2(\Z), V_m ) $  & $ \chi_{h}(\mr{GL}_2(\Z), V_m \otimes det ) $ &  $ \chi_{h}(\mr{SL}_2(\Z), V_m )$\\
\hline
\hline
$k =0 $   & $-\ell+1 $ &  $-\ell$  &  $-2\ell +1$ \\
\hline
$k =1 $  & 0   &  0 & 0 \\
\hline
$k =2 $ &  $-\ell$    & $ -\ell -1$  &  $-2\ell -1$ \\
\hline
$k =3 $ & 0  & 0   & 0  \\
\hline
$k =4$&  $-\ell$   & $ -\ell -1$   & $-2\ell -1$ \\
\hline
$k =5 $ & 0  & 0  & 0  \\
\hline
$k =6 $ & $-\ell$     & $ -\ell -1$   & $-2\ell -1$   \\
\hline
$k =7 $& 0 & 0  & 0 \\
\hline
$k =8$ &  $-\ell$    &  $ -\ell -1$  & $-2\ell -1$ \\
\hline
$k =9$ & 0 & 0  & 0 \\
\hline
$k =10$ &  $-\ell-1$     &  $ -\ell -2$   & $-2\ell -3$  \\
\hline
$k =11$ &  0 & 0  & 0 \\
\hline
\caption{Euler characteristics of $\mr{SL}_2(\Z)$ and $\mr{GL}_2(\Z)$.}\label{eulersl2gl2}
\end{longtable}
\end{center}
}

We present the following isomorphism for intuition. One can recover a simple proof by using the data of the Table~\ref{eulersl2gl2} and the Kostant formula. 
\[\begin{tabular}{llllll}
$H^{1}(\mr{SL}_2(\Z), V_m)$ 
& = &   $H^{1}_{Eis}(\mr{SL}_2(\Z), V_m)$ $\oplus$  $H^{1}_{!}(\mr{SL}_2(\Z), V_m)$  \\
& = &   $H^{1}_{Eis}(\mr{GL}_2(\Z), V_m \otimes det)$  $\oplus$  $H^{1}_{!}(\mr{SL}_2(\Z), V_m)$  \\
& = &   $H^{1}_{Eis}(\mr{GL}_2(\Z), V_m\otimes det)$  $\oplus$  $H^{1}_{!}(\mr{GL}_2(\Z), V_m\otimes det)$ $\oplus$ $H^{1}_{!}(\mr{GL}_2(\Z), V_m)$\\
&=& $H^{1}_{Eis}(\mr{GL}_2(\Z), V_m\otimes det)$  $\oplus$   $H^{1}_{!}(\mr{GL}_2(\Z), V_m)$ $\oplus$ $H^{1}_{!}(\mr{GL}_2(\Z), V_m)$
\end{tabular}
\]

\subsection{Torsion Elements in $\mr{SL}_{3}(\Z)$}\label{torsionsl3}

Following equation~\eqr{hecT} and above discussion, we know that in order to  compute $\chi_{h}(\mr{SL}_{3}(\Z), \mcV)$ with respect to the coefficient system $\mcV$, we need to know  the conjugacy classes of all torsion elements. To do that we divide the study into the possible characteristic polynomials of the representatives of these conjugacy classes, and these are:
{\begin{center}
\scriptsize\renewcommand{\arraystretch}{2.2}
\begin{longtable}{|c|c|c|c|c|c|c|c|c|}
\hline
S.No.  &  Polynomial   &  Expanded form & In $\mr{SL}_3(\Z)$ & S.No.  &  Polynomial   &  Expanded form & In $\mr{SL}_3(\Z)$ \\
\hline
\hline
1 &  $\Phi_{_1}^{^3}$  & $(x-1)^3$ & Yes & 2 & $\Phi_{_1}^{^2} \Phi_{_2}$  &$(x-1)^2 (x+1)$  & No \\  
\hline
3 & $\Phi_{_1} \Phi_{_2}^{2}$  & $(x-1) (x+1)^2$  & Yes & 4 & $\Phi_{_1}\Phi_{_3}$  & $(x-1) (x^2 + x +1)$ & Yes \\
\hline
5 & $\Phi_{_1}\Phi_{_4}$ &  $(x-1) (x^2 +1)$  & Yes & 6 & $\Phi_{_1} \Phi_{_6}$  &  $(x-1) (x^2 - x +1)$ & Yes \\
\hline
7 & $\Phi_{_2}^{^3} $  &  $(x+1)^3$  &  No & 8 & $\Phi_{_2} \Phi_{_3}$  & $(x+1) (x^2 + x +1)$ & No\\
\hline
9 & $\Phi_{_2} \Phi_{_4}$  & $(x+1) (x^2 +1)$  &No & 10 & $\Phi_{_2} \Phi_{_6}$  & $(x+1) (x^2 - x +1)$& No  \\
\hline
\caption{Torsion elements in $\mr{GL}_3(\Z)$.}\label{torsiojngl3}
\end{longtable}
\end{center}
}

Following equation~\eqr{hecT}, we compute the traces $Tr(T^{-1}| \m_{\lambda})$ of all the torsion elements $T$ in $\mr{SL}_3(\Z)$ and $\mr{GL}_3(\Z)$ with respect to highest weight coefficient system $\m_{\lambda}$ where $\lambda = m_1 \epsilon_1 + m_2 (\epsilon_1+\epsilon_2)$ and $\lambda= m_1 \epsilon_1 + m_2 (\epsilon_1+\epsilon_2)  + m_3 (\epsilon_1+\epsilon_2+\epsilon_3)$ for $\mr{SL}_3$  and $\mr{GL}_3$, respectively. 

Before moving to the next step, we will explain the reader about the use of the notation $\m_{\lambda}$. For convenience and to make the role of the coefficients $m_1, m_2$ in case of $\mr{SL}_3(\Z)$ and $m_1, m_2, m_3$  in case of $\mr{GL}_3(\Z)$ as clear as possible in the highest weight $\lambda$, we will often use  these coefficients in the subscript of the notation $\m_{\lambda}$ in place of $\lambda$, \ie we write 

\begin{eqnarray*}
\m_{\lambda} : =  \left\{\begin{array}{cccc}
& \m_{m_1, m_2}\,, & \mr{for} \quad \mr{SL}_3 \\
&\\
&  \m_{m_1, m_2, m_3}\,,& \mr{for} \quad \mr{GL}_3
\end{array}\qquad\,. \right .
\end{eqnarray*}

For any torsion element $T \in \mr{SL}_3(\Z) $, we define 
\begin{equation*}\label{eq:hnt1} H_{m}(T) := Tr(T^{-1}| \m_{m})  = Tr(T^{-1} | Sym^{m} V)  =  \sum_{a+b+c=m}  \mu_{1}^{a} \mu_{2}^{b}  \mu_{3}^{c}\,. \end{equation*}
where $\mu_1, \mu_2 $ and $\mu_3$ are the eigenvalues of $T$ and $V$ denotes the standard representation of $\mr{SL}_3$ (and $\mr{GL}_3$). Note that $\m_{m}$ above simply denotes the highest weight representation $\m_{m,0}$ of $\mr{SL}_3$. We also use the notation
\begin{equation}
H_{m_1, m_2}(\Phi) := Tr(T^{-1}| \m_{m_1, m_2})  \quad \mbox{ and } \quad  H_{m_1, m_2, m_3}(\Phi) := Tr(T^{-1}| \m_{m_1, m_2, m_3})\,, \nonumber 
\end{equation}
where $T$ is a torsion element with characteristic polynomial $\Phi$.
Therefore, 
{ \begin{center}
\scriptsize\renewcommand{\arraystretch}{2.2}
\begin{longtable}{|c|c|c|c|c|c|c|c|c|}
\hline
Case & $T$ &  $\Phi_n$ &$C(T)$   &  $\chi(C(T))$ & $Res(f,g)$ & $Res(f,g) \chi(C(T))$      \\
\hline
\hline
A & $I_3$ & $\Phi_1^{3}$& $\mr{SL}_3(\Z)$ & 0 & 0 & 0 \\
\hline
B &$ \left( \begin{array}{ccc}
1 & 0 & 0  \\
0 & -1 & 0  \\
0 & 0 & -1   \\ 
\end{array}  \right)$ &$\Phi_1 \Phi_2^{2}$  &$\mr{GL}_2(\Z)$ &  $-\frac{1}{24}$ & 4  & $-\frac{1}{6}$     \\
\hline
C & $\left( \begin{array}{ccc}
1 & 0 & 0  \\
0 & 0 & -1  \\
0 & 1 & -1  \\ 
\end{array}  \right)$ & $\Phi_1 \Phi_3$  &$C_6$    & $\frac{1}{6}$ & 3 & $\frac{1}{2}$  \\
\hline
D &$\left( \begin{array}{ccc}
1 & 0 & 0  \\
0 & 0 & 1  \\
0 & -1 & 0  \\ 
\end{array}  \right)$ &   $\Phi_1 \Phi_4$ &$C_4$    & $\frac{1}{4}$ & 2  & $\frac{1}{2}$      \\
\hline
E &$\left( \begin{array}{ccc}
1 & 0 & 0  \\
0 & 0 & 1  \\
0 & -1 & 1  \\ 
\end{array}  \right)$ & $\Phi_1 \Phi_6$   &$C_6$    & $\frac{1}{6}$ & 1 & $\frac{1}{6}$  \\
\hline
\caption{Torsion elements of $\mr{SL}_3(\Z)$.}\label{torsion-sl3}
\end{longtable}
\end{center}
}

Let $\m_{m_1, m_2}$ denote the irreducible representation of $\mr{SL}_3$ with highest weight $\lambda = m_1 \epsilon_1 + m_2 (\epsilon_1+\epsilon_2)$. Following equations~\eqr{hec} and \eqr{hnthor} we have
{
\begin{eqnarray}\label{eq:sl3chi}
 \chi_{h}(\mr{SL}_3(\Z), \m_{m_1, m_2} ) & = &  - \frac{1}{6} H_{m_1, m_2}(\Phi_1 \Phi_2^2)  + \frac{1}{2} H_{m_1, m_2}(\Phi_1 \Phi_3)  \\ \nonumber
 & + & \frac{1}{2} H_{m_1, m_2}(\Phi_1 \Phi_4)+\frac{1}{6} H_{m_1, m_2}(\Phi_1 \Phi_6).
  \end{eqnarray}
  }
 To obtain the complete information of  $ \chi_{h}(\mr{SL}_3(\Z), \m_{m_1, m_2} ) $, let us compute the $H_{m_1, m_2}(\Phi_1 \Phi_2^2) $,  $H_{m_1, m_2}(\Phi_1 \Phi_3) $, $H_{m_1, m_2}(\Phi_1 \Phi_4) $ and $H_{m_1, m_2}(\Phi_1 \Phi_6) $. One could do this by using the Weyl character formula as defined in Chapter~24 of~\cite{FultonHarris},
\begin{equation*} H_{m_1, m_2}(\Phi_1 \Phi_k) = det \tmt{H_{m_1+m_2}{(\Phi_1 \Phi_k)}}{H_{m_1+m_2+1}(\Phi_1 \Phi_k)}{H_{m_2 -1}(\Phi_1 \Phi_k)}{H_{m_2}(\Phi_1 \Phi_k)},\end{equation*}  
but we will use another argument to calculate these traces. For that we consider the case $\mr{GL}_3(\Z)$ and obtain the needed results as a corollary. 
 
 \begin{lema}
 Let $\xi_k = e^{\frac{2\pi i}{k}}$, then $$H_{m_1, m_2, m_3}(\Phi_1 \Phi_k) = \sum_{p_1 = m_2+m_3}^{m_2+m_2+m_3} \sum_{p_2 = m_3} ^{m_2+m_3} \sum_{q = p_2}^{p_1} \xi_k^{2q - (p_1 + p_2)},$$
for $k = 3, 4, 6$ and $$H_{m_1, m_2, m_3}(\Phi_1 \Phi_2^2) = \sum_{p_1 = m_2+m_3}^{m_1+m_2+m_3} \sum_{p_2 = m_3} ^{m_2+m_3} \sum_{q = p_2}^{p_1} \xi_2^{2q - (p_1 + p_2)}.$$
\end{lema}
 
\begin{proof}
We use the description of $\m_{m_1, m_2, m_3}$ given in~\cite{Gelfand50}. In particular, one has a basis $$\left\{ L\left(  \begin{matrix}{p_1}{p_2}\\{q} \end{matrix}  \right) \mid m_1+m_2+m_3 \geq p_1 \geq m_2+m_3, m_2+m_3 \geq p_2 \geq m_3, p_1 \geq q \geq p_2 \right\}$$
such that under the action of $\mathfrak{gl}_3$, $$(E_{1, 1} - E_{2, 2})\left(L\left( \begin{matrix}{p_1}{p_2}\\{q} \end{matrix}   \right)\right) = (2q - (p_1 + p_2))L\left( \begin{matrix}{p_1}{p_2}\\{q} \end{matrix}   \right).$$

If we denote by $\rho_{m_1, m_2, m_3}$ the representation corresponding to $\m_{m_1, m_2, m_3}$ then the diagram 
$$\xymatrixcolsep{6pc}\xymatrix{
\mr{GL}_3(\C) \ar@{<-}^{exp}[d]  \ar@{->}^{\rho_{m_1, m_2, m_3}}[r] & \mr{GL}(\m_{m_1, m_2, m_3}) \ar@{<-}^{exp}[d] \\
\mathfrak{gl}_3(\C)  \ar@{->}^{d \rho_{m_1, m_2, m_3}}[r] & \mathfrak{gl}(\m_{m_1, m_2, m_3}) }$$ 
is commutative. Therefore 
$$\left( \begin{array}{rrr}
\xi_k &  0 & 0  \\
0 & \xi_k^{-1} & 0  \\
0 &  0 & 1  \\ 
\end{array}  \right) L\left( \begin{matrix}{p_1}{p_2}\\{q} \end{matrix}   \right) = \xi_k^{2q - (p_1 + p_2)} L\left( \begin{matrix}{p_1}{p_2}\\{q} \end{matrix}   \right) $$
and the result follows simply by using the fact that $$H_{m_1, m_2, m_3}(\Phi_1 \Phi_k) = Tr\left(\left( \begin{array}{rrr}
\xi_k &  0 & 0  \\
0 & \xi_k^{-1} & 0  \\
0 &  0 & 1  \\ 
\end{array}  \right), \m_{m_1, m_2, m_3}\right).$$ 
\end{proof} 

We denote $C_k(p_1, p_2) = \sum_{q = p_2}^{p_1} \xi_k^{2q - (p_1 + p_2)}$, for $k = 2, 3, 4, 6$. By using the fact that $$\xi_k^{2(\frac{p_1 + p_2}{2} - j) - (p_1 + p_2) } = \xi_k^{-2j} = (\xi_k^{2j})^{-1} = (\xi_k^{2(\frac{p_1 + p_2}{2} + j) - (p_1 + p_2) })^{-1} \quad \forall j \in \mathbb{Z},$$ one has that
\begin{equation*}\label{Formula}
C_k(p_1, p_2) = \left\{\begin{array}{cccc}  
& 1 + \sum_{q = p_2}^{\frac{p_1+p_2}{2}-1} \left(\xi_k^{2q - (p_1 + p_2)} + (\xi_k^{2q - (p_1 + p_2)})^{-1} \right)\,,  &  p_1 \equiv p_2 (mod \, 2)  \\ 
&\\
& \sum_{q = p_2}^{\frac{p_1+p_2-1}{2}} \left(\xi_k^{2q - (p_1 + p_2)} + (\xi_k^{2q - (p_1 + p_2)})^{-1} \right) \,, & otherwise\\
\end{array}\qquad\,. \right.  
 \end{equation*}

\begin{lema}
\begin{equation*}
C_6(p_1, p_2) = \left\{\begin{array}{rrrrrr}  
& 1 \,,  &  p_1 - p_2 \equiv 0 (mod \, 6)  \\ 
&\\
& 1 \,,  &  p_1 - p_2 \equiv 1 (mod \, 6)  \\
&\\
& 0 \,,  &  p_1 - p_2 \equiv 2 (mod \, 6)  \\ 
&\\
& -1 \,,  &  p_1 - p_2 \equiv 3 (mod \, 6)  \\
&\\
& -1 \,,  &  p_1 - p_2 \equiv 4 (mod \, 6)  \\ 
&\\
& 0 \,,  &  p_1 - p_2 \equiv 5 (mod \, 6)  \\
\end{array}\qquad\,, \right .  
 \end{equation*}
\end{lema}
\begin{proof}
One can check that 
 \begin{equation*}
\xi_6^{\ell} + \xi_6^{-\ell}= \left\{\begin{array}{rrrrrr}  
& 2 \,,  &  \ell \equiv 0 (mod \, 6)  \\ 
&\\
& 1 \,,  &  \ell \equiv 1 (mod \, 6)  \\
&\\
& -1 \,,  &  \ell \equiv 2 (mod \, 6)  \\ 
&\\
& -2 \,,  &  \ell \equiv 3 (mod \, 6)  \\
&\\
& -1 \,,  &  \ell \equiv 4 (mod \, 6)  \\ 
&\\
& 1 \,,  &  \ell \equiv 5 (mod \, 6)  \\
\end{array}\qquad\,, \right .  
 \end{equation*}
This implies that for every integer $\ell$, $$\sum_{j=1}^3 \xi_6^{\ell +2j} + \xi_6^{-(\ell+2j)} = 0,$$ in other words, the sum of three consecutive terms in the formula for $C_6(p_1, p_2)$ is zero and $C_6(p_1, p_2)$ only depends on $p_1-p_2$ modulo $6$.
\end{proof} 

Following the similar procedure we deduce the values of $C_4(p_1, p_2) $ and $C_3(p_1, p_2)$ which we summarize in the following lemma.

\begin{lema}
\begin{equation*}
C_4(p_1, p_2) = \left\{\begin{array}{rrrrrr}  
& 1 \,,  &  p_1 - p_2 \equiv 0 (mod \, 4)  \\ 
&\\
& 0 \,,  &  p_1 - p_2 \equiv 1 (mod \, 4)  \\
&\\
& -1 \,,  &  p_1 - p_2 \equiv 2 (mod \, 4)  \\ 
&\\
& 0 \,,  &  p_1 - p_2 \equiv 3 (mod \, 4)  \\
\end{array}\qquad\, \right.,  
 \end{equation*}
and
\begin{equation*}
C_3(p_1, p_2) = \left\{\begin{array}{rrrrrr}  
& 1 \,,  &  p_1 - p_2 \equiv 0 (mod \, 3)  \\ 
&\\
& -1 \,,  &  p_1 - p_2 \equiv 1 (mod \, 3)  \\
&\\
& 0 \,,  &  p_1 - p_2 \equiv 2 (mod \, 3)  \\
\end{array}\qquad\,. \right .  
 \end{equation*}
\end{lema}

\begin{rmk} For $k=3, 4,6$, the sum of the $C_k(p_1, p_2)$ for the different possible congruences of $p_1 - p_2$ modulo $k$ is zero, and this implies that
$$H_{m_1, m_2, m_3}(\Phi_1 \Phi_k) = \sum_{p_1 = m_2+m_3}^{m_1+m_2+m_3} \sum_{p_2 = m_3} ^{m_2+m_3} C_k(p_1, p_2)$$
depends only on the congruences of $m_1, m_2$ and $m_3$ modulo $k$.
\end{rmk}
Following the  above discussion, it is straightforward to prove the following
 
\begin{lema}\label{M346}
For $m_1, m_2, m_3 \in \N$ and $k =3, 4, 6$,  let $i$ and $j$ be $(m_1 \mbox{ mod } k) + 1$ and $(m_2 \mbox{ mod } k)+1$ respectively. Then $H_{m_1, m_2, m_3}(\Phi_1 \Phi_k)$ is the $(i, j)$-entry of the matrix $M_k$ where 
$$M_6=\left( \begin{array}{rrrrrr}
1 &  2 &  2 &  1 &  0 & 0 \\
2 &  3 &  2 &  0 & -1 & 0\\
2 &  2 &  0 & -2 & -2 & 0  \\ 
1 &  0 & -2 & -3 & -2 & 0 \\ 
0 & -1 & -2 & -2 & -1 & 0  \\ 
0 &  0 &  0 &  0 &  0 & 0 \\ 
\end{array}  \right),
M_4 = \left( \begin{array}{rrrr}
1 &  1  &  0 & 0 \\
1 &  0  & -1 & 0\\
0 & -1  & -1 & 0  \\ 
0 &  0  &  0 & 0  \\ 
\end{array}  \right) \mbox{ and }
M_3 = \left( \begin{array}{rrr}
1 &  0 & 0  \\
0 & -1 & 0  \\
0 &  0 & 0  \\ 
\end{array}  \right).$$
\end{lema} 
 
 \begin{lema}\label{M2} For $m_1, m_2, m_3 \in \N$,  let $i$ and $j$ be $(m_1 \mbox{ mod } 2) + 1$ and $ (m_2 \mbox{ mod } 2)+1$ respectively. Then $H_{m_1, m_2, m_3}(\Phi_1 \Phi_2^2)$ is the $(i, j)$-entry of the matrix $M_2$, where
 $$ M_2 = \left( \begin{array}{ccc}
  1 + \frac{m_1+m_2}{2} & - \frac{m_2 + 1}{2}  \\
  &\\
- \frac{m_1 + 1}{2}    &  0  \\
\end{array}  \right)$$
\end{lema}

\begin{proof}
We have
\begin{equation*}\label{Formula1}
C_2(p_1, p_2) = \left\{\begin{array}{cccc}  
& p_1 - p_2 + 1\,,  &  p_1 \equiv p_2 (mod \, 2)  \\ 
&\\
& -(p_1 - p_2 + 1) \,, & otherwise\\
\end{array}\qquad\,, \right .  
 \end{equation*}

and $$H_{m_1, m_2, m_3}(\Phi_1 \Phi_2^2) = \sum_{p_1 = m_2+m_3}^{m_1+m_2+m_3} \sum_{p_2 = m_3} ^{m_2+m_3} (-1)^{p_1 + p_2}(p_1 - p_2 + 1).$$

We now make a case by case study with respect to the parity of $m_1$ and $m_2$. If $m_2$ is even then for a fixed $p_1$, 
$$\sum_{p_2 = m_3} ^{m_2+m_3} (-1)^{p_1 + p_2}(p_1 - p_2 + 1) = (-1)^{p_1 + m_3}\bigg(p_1 + 1 - \frac{m_2}{2} -m_3\bigg).$$
Moreover, If $m_1$ is even then
$$H_{m_1, m_2, m_3}(\Phi_1 \Phi_2^2) = \sum_{p_1 = m_2+m_3}^{m_1+m_2+m_3} (-1)^{p_1 + m_3}\bigg(p_1 + 1 - \frac{m_2}{2} - m_3\bigg) = 1 + \frac{m_1+m_2}{2}.$$
On the other hand, if $m_1$ is odd then
$$H_{m_1, m_2, m_3}(\Phi_1 \Phi_2^2) = \sum_{p_1 = m_2+m_3}^{m_1+m_2+m_3} (-1)^{p_1 + m_3}\bigg(p_1 + 1 - \frac{m_2}{2}\bigg) = - \frac{m_1 + 1}{2}.$$
Now, if $m_2$ is odd then for a fixed $p_1$, 
$$\sum_{p_2 = m_3} ^{m_2+m_3} (-1)^{p_1 + p_2}(p_1 - p_2 + 1) = (-1)^{p_1 + m_3}\bigg(\frac{m_2 + 1}{2}\bigg)$$
and this depends only on the parity of $p_1$. Hence, 
\begin{equation*}
H_{m_1, m_2, m_3}(\Phi_1 \Phi_2^2) = \left\{\begin{array}{cccc}  
& - \frac{m_2 + 1}{2} \,,  &  p_1 \equiv 0 (mod \, 2)  \\ 
&\\
& 0 \,, & otherwise\\
\end{array}\qquad\,. \right .  
 \end{equation*}
\end{proof}

\subsection{Euler Characteristic of $\mr{SL}_3(\Z)$ with respect to the highest weight representations}\label{sl3euler}
We compute the $ \chi_{h}(\mr{SL}_3(\Z), \m_{m_1, m_2} ) $ in the following table by  computing each factor of the above equation~\eqr{sl3chi} up to modulo 12, which is achieved simply by following the discussion of previous Subsection~\ref{torsionsl3} and more explicitly from Lemma~\ref{M346} and Lemma~\ref{M2}. All these values are encoded in the following table consisting of $144$ entries where rows run from $0 \leq i \leq 11$ representing $m_1 \equiv i (mod \, 12)$ and columns runs through  $ 0 \leq j \leq 11$ representing $m_2 \equiv j (mod \, 12)$. To accommodate the data with the available space, the table has been divided into two different tables of order $12 \times 6$ each. In the first table (Table 8) $j$ runs from $0 (mod \, 12)$ to $5 (mod \, 12)$ and in the second table (Table 9) from $ 6(mod\, 12)$ to  $11 (mod\, 12)$ and in both tables $i$ runs from $ 0 (mod\, 12)$ to $ 11 (mod\, 12)$.

Once the entries of the table are computed, we get complete information about the Euler characteristics of $\mr{SL}_3(\Z)$ which is summarized in the following
\begin{thm}
\label{euler}
The Euler characteristics of $\mr{SL}_3(\Z)$ with coefficient in any highest weight representation $\m_{m_1,m_2}$, can be described by one of the following four cases, depending on the parity of $m_1$ and $m_2$. More precisely,
\begin{equation}
\chi_{h}(\mr{SL}_3(\Z), \m_{m_1, m_2} ) = \left\{\begin{array}{cccc}  
& -1-\dim S_{m_1+2}-\dim S_{m_2+2} \,,  &  m_1, m_2 \,\, \mr{both}\, \mr{even}  \\ 
&\\
& -\dim S_{m_1+2}+\dim S_{m_1+m_2+3} \,, & m_1 \, \mr{even},\, m_2 \, \mr{odd} \\
&\\
& -\dim S_{m_2+2}+\dim S_{m_1+m_2+3} \,, & m_1 \, \mr{odd}, \, m_2 \, \mr{even}\\
&\\
& 0 \,, &  m_1, m_2 \, \,\mr{both\,\, odd}\\
\end{array}\quad\,, \right .  
 \end{equation}

where $S_{m+2}$, as described earlier in Section 5.1 by equation~\eqr{blah}, is the space of holomorphic cusp forms of weight $m+2$ for $\rm{SL}_2(\Z)$, and for $m=0$ we define $\dim S_2=-1$. 
\end{thm}

For the reader's convenience, the dimension of the space of cusp forms $S_{m+2}$ is given by 
$$\dim S_{12\ell+2+i}=
\left\{
\begin{tabular}{lll}
$\ell-1$ & if $i=0$\\
$\ell$ & if $i=2,4,6,8$\\
$\ell+1$ &if $i=10$\\
$0$ & if $i$ is odd
\end{tabular}.
\right.$$

{
\afterpage{%
    \clearpage
\small{
\begin{landscape}
 \centering 
\vfill
\begin{table}[!htbp]
\begin{threeparttable}
\begin{tabular}{l*{6}{c}}
\vspace{0.8cm}
$-\frac{1}{12}(m_1+m_2)+1$ &$ \frac{1}{12} (m_2-1)+1$  &$ -\frac{1}{12}(m_1+m_2 -2) $ & $\frac{1}{12} (m_2-3)+1 $ & $-\frac{1}{12} (m_1+m_2-4) $ & $\frac{1}{12} (m_2-5)+1$ \\

\vspace{0.8cm}

$\frac{1}{12}(m_1-1) +1$ & 0 & $\frac{1}{12}(m_1-1) $&  0 &$ \frac{1}{12}(m_1-1)$  & 0  \\

\vspace{0.8cm}

$-\frac{1}{12} (m_1+m_2-2)$ & $\frac{1}{12}(m_2-1) $ & $-\frac{1}{12} (m_1+m_2-4)-1$ & $\frac{1}{12}(m_2-3) $ & $-\frac{1}{12} (m_1+m_2-6)-1$ & $\frac{1}{12}(m_2-5)$ \\

\vspace{0.8cm}


$\frac{1}{12}(m_1-3) +1$ & 0 & $\frac{1}{12}(m_1-3)$ & 0  & $\frac{1}{12}(m_1-3)$ & 0  \\

\vspace{0.8cm}


$-\frac{1}{12} (m_1+m_2-4)$ & $ \frac{1}{12}(m_2-1)$  & $-\frac{1}{12} (m_1+m_2-6)-1$  & $\frac{1}{12}(m_2-3) $  & $-\frac{1}{12} (m_1+m_2-8)-1$ & $\frac{1}{12}(m_2-5) +1$ \\

\vspace{0.8cm}


$\frac{1}{12}(m_1-5) +1$ & 0 & $\frac{1}{12}(m_1-5)$ & 0  & $\frac{1}{12}(m_1-5)+1$ & 0  \\

\vspace{0.8cm}


$-\frac{1}{12}(m_1+m_2-6)$ & $\frac{1}{12} (m_2-1)$  & $-\frac{1}{12}(m_1+m_2 -8)-1$  & $\frac{1}{12} (m_2-3)+1$  & $-\frac{1}{12} (m_1+m_2-10)-1 $ & $\frac{1}{12} (m_2-5)$  \\

\vspace{0.8cm}

$\frac{1}{12}(m_1-7) +1$ & 0 & $\frac{1}{12}(m_1-7)+1$ &  0 & $\frac{1}{12}(m_1-7)$  & 0  \\

\vspace{0.8cm}

$-\frac{1}{12} (m_1+m_2-8)$ & $\frac{1}{12}(m_2-1) +1$ & $-\frac{1}{12} (m_1+m_2-10)-1$ & $\frac{1}{12}(m_2-3) $ & $-\frac{1}{12} (m_1+m_2-12)-1$ & $ \frac{1}{12}(m_2-5)+1$ \\

\vspace{0.8cm}


$\frac{1}{12}(m_1-9) +2$ & 0 & $\frac{1}{12}(m_1-9)$ & 0  & $\frac{1}{12}(m_1-9)+1$ & 0 \\

\vspace{0.8cm}

$-\frac{1}{12} (m_1+m_2-10)-1$ & $\frac{1}{12}(m_2-1)-1$  & $-\frac{1}{12} (m_1+m_2-12)-2$  & $\frac{1}{12}(m_2-3)$   & $-\frac{1}{12} (m_1+m_2-14)-2$ & $\frac{1}{12}(m_2-5)$  \\

\vspace{0.8cm}

$\frac{1}{12}(m_1-11) +1$ & 0 &$ \frac{1}{12}(m_1-11)+1$ & 0  & $\frac{1}{12}(m_1-11)+1$ & 0 \\

\end{tabular}
\vspace{0.4cm}
\caption{First table comprising 72 values of $\chi_{h}(\mr{SL}_3(\Z), \m_{m_1, m_2} )$ with rows $0 \leq m_1 \leq 11$ and columns $0 \leq m_2 \leq 5$ both $(mod\, 12)$.}
\end{threeparttable}
\end{table}
\vfill
\end{landscape}
 \clearpage
}
}
}

\afterpage{%
\clearpage
\small{
\begin{landscape}
\vfill
\begin{table}[!htbp]
\begin{threeparttable}
\begin{tabular}{l*{12}{c}}
\vspace{0.8cm}
 $ -\frac{1}{12}(m_1+m_2-6)$ & $\frac{1}{12} (m_2-7)+1$  & $-\frac{1}{12}(m_1+m_2 -8) $ & $\frac{1}{12} (m_2-9)+2 $ & $-\frac{1}{12} (m_1+m_2-10)-1$ & $\frac{1}{12} (m_2-11)+1$ \\

\vspace{0.8cm}

 $\frac{1}{12}(m_1-1) $ & 0 & $\frac{1}{12}(m_1-1)+1 $ &  0 & $ \frac{1}{12}(m_1-1)-1$  & 0 \\

\vspace{0.8cm}

 $-\frac{1}{12} (m_1+m_2-8) -1$ & $\frac{1}{12}(m_2-7)+1$  & $-\frac{1}{12} (m_1+m_2-10)-1$ & $\frac{1}{12}(m_2-9)$  & $-\frac{1}{12} (m_1+m_2-12)-2$ & $\frac{1}{12}(m_2-11)+1$\\

\vspace{0.8cm}


 $\frac{1}{12}(m_1-3) +1$ & 0 & $\frac{1}{12}(m_1-3)$ & 0  & $\frac{1}{12}(m_1-3)$ & 0 \\

\vspace{0.8cm}


 $-\frac{1}{12} (m_1+m_2-10)-1$ & $\frac{1}{12}(m_2-7)$  & $-\frac{1}{12} (m_1+m_2-12)-1$  & $\frac{1}{12}(m_2-9)+1$  & $-\frac{1}{12} (m_1+m_2-14)-2$ & $\frac{1}{12}(m_2-11) +1$\\

\vspace{0.8cm}


 $\frac{1}{12}(m_1-5)$  & 0 & $\frac{1}{12}(m_1-5)+1$ & 0  & $ \frac{1}{12}(m_1-5) $ & 0 \\

\vspace{0.8cm}


 $-\frac{1}{12}(m_1+m_2-12)-1$ & $ \frac{1}{12} (m_2-7)+1$  & $-\frac{1}{12}(m_1+m_2 -14)-1$ & $\frac{1}{12} (m_2-9)+1 $ & $-\frac{1}{12} (m_1+m_2-16)-2$ & $\frac{1}{12} (m_2-11)+1$  \\

\vspace{0.8cm}

 $\frac{1}{12}(m_1-7) +1$ & 0 & $\frac{1}{12}(m_1-7)+1$ &  0 & $\frac{1}{12}(m_1-7) $ & 0\\

\vspace{0.8cm}

 $-\frac{1}{12} (m_1+m_2-14)-1$ & $\frac{1}{12}(m_2-7) +1$ & $-\frac{1}{12} (m_1+m_2-16)-1$ & $\frac{1}{12}(m_2-9)+1$  & $-\frac{1}{12} (m_1+m_2-18)-2$ & $\frac{1}{12}(m_2-11)+1$\\

\vspace{0.8cm}


$\frac{1}{12}(m_1-9) +1$ & 0 & $\frac{1}{12}(m_1-9)+1$ & 0  & $\frac{1}{12}(m_1-9)$ & 0\\

\vspace{0.8cm}

 $-\frac{1}{12} (m_1+m_2-16)-2$ & $\frac{1}{12}(m_2-7)$  & $-\frac{1}{12} (m_1+m_2-18)-2$  & $\frac{1}{12}(m_2-9)$   & $-\frac{1}{12} (m_1+m_2-20)-3$ & $\frac{1}{12}(m_2-11)+1$\\

\vspace{0.8cm}

 $\frac{1}{12}(m_1-11) +1$ & 0 & $\frac{1}{12}(m_1-11)+1$ & 0  &$ \frac{1}{12}(m_1-11)+1 $& 0\\

\end{tabular}
\vspace{0.4cm}
\caption{Second half of the total 144 entries, comprising 72 values of $\chi_{h}(\mr{SL}_3(\Z), \m_{m_1, m_2})$ with rows $0 \leq m_1 \leq 11$ and columns $6 \leq m_2 \leq 11$ both $(mod \,12)$.}
\end{threeparttable}
\end{table}
\vfill
\end{landscape}
 \clearpage
}
}

\subsection{Euler Characteristic of $\mr{GL}_3(\Z)$ with respect to the highest weight representations}  This subsection is merely an example to reveal the fact that the results obtained for $\mr{SL}_3(\Z)$ can easily be extended to $\mr{GL}_3(\Z)$. However, This can also be easily concluded by using the Lemma~\ref{sl3togl3} which appears later in Section 6. 

Let $T$ be any torsion element of $\mr{GL}_3(\Z)$. Then  $H_{m}(-T) = (-1)^{m} H_m(T)\,.$ Therefore 
\begin{equation*}H_m (T) + H_m (-T) = \left\{\begin{array}{cccc}  
& 2 H_m(T)\,,  &  m \equiv 0 (mod \, 2)  \\ 
&\\
& 0 \,, & m \equiv 1 (mod \, 2)\\
\end{array}\qquad\,. \right .  
 \end{equation*}
 
For any $T \in \mr{SL}_3(\Z)$, $C_{\mr{GL}_{3}(\Z)}(T) = \{\pm I\} \times C_{\mr{SL}_{3}(\Z)}(T)$. This implies that $$\chi_{orb}(C_{\mr{GL}_{3}(\Z)}(T)) = \frac{1}{2} \chi_{orb}(C_{\mr{SL}_{3}(\Z)}(T))\,.$$
This gives 
\begin{equation*}
\chi_{orb}(C_{\mr{GL}_{3}(\Z)}(T)) H_m (T) + \chi_{orb}(C_{\mr{GL}_{3}(\Z)}(-T))  H_m (-T) =  \left\{\begin{array}{cccc}  
& \chi_{orb}(C_{\mr{SL}_{3}(\Z)}(T)) H_m(T)\,,  &  m \equiv 0 (mod \, 2)  \\ 
& \\
& 0 \,, & m \equiv 1 (mod \, 2)\\
\end{array}. \right . 
\end{equation*}

Therefore, 

\begin{equation*}
\chi_{h}(\mr{GL}_{3}(\Z), Sym^{m} V) =  \left\{\begin{array}{cccc}  
& \chi_{h}(\mr{SL}_{3}(\Z), Sym^{m} V)\,,  &  m \equiv 0 (mod \, 2)  \\ 
& \\
& 0 \,, & m \equiv 1 (mod \, 2)\\
\end{array}. \right . 
\end{equation*}

More generally, following the Weyl character formula, for any torsion element $T \in \mr{GL}_3(\Z)$,  we write
$$H_{m_1, m_2, m_3}(-T) = (-1)^{m_1+2m_2+3m_3} H_{m_1, m_2, m_3}(T) = (-1)^{m_1 + m_3} H_{m_1, m_2, m_3}(T) \,.$$ This implies that
\begin{equation*}
\chi_{h}(\mr{GL}_{3}(\Z), \m_{m_1, m_2, m_3}) =  \left\{\begin{array}{cccc}  
& \chi_{h}(\mr{SL}_{3}(\Z), \m_{m_1, m_2})\,,  &  m_1 + m_3 \equiv 0 (mod \, 2)  \\ 
& \\
& 0 \,, & m_1 + m_3 \equiv 1 (mod \, 2)\\
\end{array}. \right . 
\end{equation*}




\section{Eisenstein Cohomology} 
In this section, by using the information obtained about boundary cohomology and Euler characteristic of $\mr{SL}_{3}(\Z)$, we discuss the Eisenstein cohomology with coefficients in $\m_\lambda$. We define the Eisenstein cohomology as the image of the restriction morphism to the boundary cohomology
\begin{equation}\label{eq:eis}
r : H^{\bullet}(\rS, \tm_\lambda)   \longrightarrow H^{\bullet}({\partial \overline{\rS}}, \tm_\lambda).
\end{equation}
In general, one can find the definition of Eisenstein cohomology as a certain subspace of $H^{\bullet}(\rS, \tm_\lambda)$ that is a complement of a subspace of the interior cohomology. It is known that the interior cohomology $H_!^\bullet(\rS, \tm_\lambda)$ is the kernel of the restriction morphism $r$. More precisely, we can simply consider the following happy scenario where the following sequence is exact.
\begin{equation*}
 0 \longrightarrow H_{!}^\bullet(\rS, \tm_\lambda) \longrightarrow  H^\bullet(\rS, \tm_\lambda) \xrightarrow{r} H^\bullet_{Eis}(\rS, \tm_\lambda )\longrightarrow 0 \,.
 \end{equation*}
 
To manifest the importance of the ongoing work and the complications involved, we refer the interested reader to an important article~\cite{LeeSch82} of Lee and Schwermer.

\subsection{A summary of boundary cohomology}

For further exploration, we summarize the discussion of boundary cohomology of $\mr{SL}_3(\Z)$ carried out  in Section~\ref{bdsl3} in the form of  following theorem.

\begin{thm}\label{bdcohsl3} For $\lambda = m_1\varepsilon_1 + m_2(\varepsilon_1+\varepsilon_2)$, the boundary cohomology of the orbifold $\mr{S}$ of the arithmetic group $\mr{SL}_3(\Z)$ with coefficients in the highest  weight representation $\m_\lambda$ is described as follows.
\begin{enumerate} 
\item Case 1 : $m_1 = m_2 = 0$  then 
\begin{equation*}
H^q(\partial \mr{S}, \tm_\lambda) =\left\{\begin{array}{cccc}  
&\mathbb{Q} \quad \mr{for}\quad  q=0,4 \\ 
&\\
&0  \qquad \mr{otherwise}\\
\end{array}\qquad\,. \right .
\end{equation*}

\item  Case 2 :  $m_1=0$ and $m_2 \neq 0$, $m_2 $ even

\begin{eqnarray*}
H^q(\partial \rS, \tm_\lambda) 
& = &
\left\{\begin{array}{cccc}  
& H_{!}^{1}(\rS^{M_1}, \tm_{e\cdot\lambda}) \,, & q=1 \\ 
&\\
& H_{!}^{1}(\rS^{M_2}, \tm_{s_2s_1\cdot\lambda}) \,, & q=3 \\
&\\
&0  \,, & \mr{otherwise}\\
\end{array}\,, \right .\\
 & = &
\left\{\begin{array}{cccc}  
& S_{m_{2} + 2} \,, & q=1 \\ 
&\\
& S_{m_2+2}\,, & q=3 \\
&\\
&0  \,, & \mr{otherwise}\\
\end{array}\,. \right .
\end{eqnarray*}

\item  Case 3 :  $m_1 \neq 0$, $m_1$ even and $m_2 =0$

\begin{eqnarray*}
H^q(\partial \rS, \tm_\lambda) 
& = & 
\left\{\begin{array}{cccc}  
& H_{!}^{1}(\rS^{M_2}, \tm_{e\cdot\lambda}) \,, & q=1 \\ 
&\\
& H_{!}^{1}(\rS^{M_1}, \tm_{s_1s_2\cdot\lambda}) \,, & q=3 \\
&\\
&0  \,, & \mr{otherwise}\\
\end{array}\,, \right. \\ 
& = &  
\left\{\begin{array}{cccc}  
& S_{m_1+2} \,, & q=1 \\ 
&\\
& S_{m_1+2} \,, & q=3 \\
&\\
&0  \,, & \mr{otherwise}\\
\end{array}\,. \right .
\end{eqnarray*}

\item  Case 4 :  $m_1 \neq 0$, $m_1$ even and $m_2 \neq 0$, $m_2$ even, then

\begin{eqnarray*}
H^q(\partial \rS, \tm_\lambda) 
& = &
\left\{\begin{array}{cccc}  
& \Q \oplus H_{!}^{1}(\rS^{M_1}, \tm_{e\cdot\lambda})\oplus H_{!}^{1}(\rS^{M_2}, \tm_{e\cdot\lambda}) \,, & q=1 \\ 
&\\
&  H_{!}^1(\rS^{M_1}, \tm_{s_1s_2\cdot\lambda}) \oplus H_{!}^1(\rS^{M_2}, \tm_{s_2s_1\cdot\lambda}) \oplus \Q \,, & q=3 \\
&\\
&0  \,, & \mr{otherwise}\\
\end{array}\,, \right . \\
& = &
\left\{\begin{array}{cccc}  
& \Q \oplus S_{m_1+2}\oplus S_{m_2+2} \,, & q=1 \\ 
&\\
& \Q \oplus S_{m_1+2}\oplus S_{m_2+2} \,, & q=3 \\
&\\
&0  \,, & \mr{otherwise}\\
\end{array}\,. \right .
\end{eqnarray*}

\item   Case 5 :  $m_1 \neq 0$, $m_1$ even and $m_2 $ odd, then

\begin{eqnarray*}
H^q(\partial \rS, \tm_\lambda) & = & 
\left\{\begin{array}{cccc}  
&H^{1}_{!}(\rS^{M_2}, \tm_{e\cdot\lambda})\,,  &  q=1 \\ 
&\\
& H^{1}_{!}(\rS^{M_1}, \tm_{s_1 \cdot \lambda}) \oplus H^{1}_{!}(\rS^{M_2}, \tm_{s_2 \cdot \lambda}) \,, & q=2\\
&\\
& H^{1}_{!}(\rS^{M_1}, \tm_{s_1s_2 \cdot \lambda})\,, & q=3\\
&\\
& 0 \,, & \mr{otherwise}
\end{array} \,, \right . \\ 
& = & 
\left\{\begin{array}{cccc}  
& S_{m_1+2}\,,  &  q=1 \\ 
&\\
& S_{m_1+m_2+3} \oplus S_{m_1+m_2+3}\,, & q=2\\
&\\
& S_{m_1+2}\,, & q=3\\
&\\
& 0 \,, & \mr{otherwise}
\end{array} \,. \right .
 \end{eqnarray*}

\item  Case 6 :  $m_1 = 0$ and $m_2$ odd, then

\begin{eqnarray*}
H^q(\partial \rS, \tm_\lambda) & = &  
\left\{\begin{array}{cccc}
& H_{!}^1(\rS^{M_1}, \tm_{s_1\cdot\lambda}) \oplus H_{!}^1(\rS^{M_2}, \tm_{s_2\cdot\lambda}) \oplus \Q \oplus \Q \,, & q=2 \\
&\\
& 0 \,,& \mr{otherwise}
\end{array}\qquad\,, \right .\\
& = & 
\left\{\begin{array}{cccc}
& S_{m_1+m_2+3} \oplus S_{m_1+m_2+3} \oplus \Q \oplus \Q \,, & q=2 \\
&\\
& 0 \,,& \mr{otherwise}
\end{array}\qquad\,. \right .\\
\end{eqnarray*}

\item  Case 7 :  $m_1 $ odd and $m_2 =0$
\begin{eqnarray*}
H^q(\partial \rS, \tm_\lambda)  & = &  \left\{\begin{array}{cccc}
& H_{!}^1(\rS^{M_2}, \tm_{s_2\cdot\lambda}) \oplus H_{!}^1(\rS^{M_1}, \tm_{s_1\cdot\lambda}) \oplus \Q \oplus \Q \,, & q=2 \\
& \\
& 0 \,,& \mr{otherwise}
\end{array}\qquad\,, \right .\\
& = & \left\{\begin{array}{cccc}
& S_{m_1+m_2+3} \oplus S_{m_1+m_2+3}\oplus \Q \oplus \Q \,, & q=2 \\
& \\
& 0 \,,& \mr{otherwise}
\end{array}\qquad\,. \right .
\end{eqnarray*}

\item  Case 8 :  $m_1$ odd, $m_2 \neq 0$ even, then
\begin{eqnarray*}\label{eq:e520}
H^q(\partial \rS, \tm_\lambda) & = & \left\{\begin{array}{cccc}  
&H^{1}_{!}(\rS^{M_1}, \tm_{e\cdot\lambda})\,,  &  q=1 \\ 
&\\
& H^{1}_{!}(\rS^{M_1}, \tm_{s_1 \cdot \lambda}) \oplus H^{1}_{!}(\rS^{M_2}, \tm_{s_2 \cdot \lambda}) \,, & q=2\\
&\\
& H^{1}_{!}(\rS^{M_2}, \tm_{s_2s_1 \cdot \lambda})\,, & q=3\\
&\\
& 0 \,, & \mr{otherwise}
\end{array}\qquad\,, \right .\\
& = & \left\{\begin{array}{cccc}  
&S_{m_2+2}\,,  &  q=1 \\ 
&\\
& S_{m_1+m_2+3} \oplus S_{m_1+m_2+3} \,, & q=2\\
&\\
& S_{m_2+2}\,, & q=3\\
&\\
& 0 \,, & \mr{otherwise}
\end{array}\qquad\,. \right .\\
\end{eqnarray*}

\item  Case 9 :  $m_1$ odd and $m_2$ odd, then
\begin{equation*}
H^{q}(\partial \rS, \widetilde{\mathcal{M}}_\lambda) = 0 \,, \qquad \forall q \,.
\end{equation*}
\end{enumerate}
\end{thm}

\par 
Observe that at this point we have explicit formulas to determine the cohomology of the boundary.

\subsection{Poincar\'e Duality}

Let $\m_\lambda^\ast$ denote the dual representation of $\m_\lambda$. $\m_\lambda^\ast$ is in fact the irreducible representation $\m_{\lambda^\ast}$ associated to the highest weight $\lambda^\ast = -w_0(\lambda)$, where $w_0$ denotes the longest element in the Weyl group. One has the natural pairings (see~\cite{HR2015}) $$H^\bullet(\rS_\Gamma, \tm_\lambda) \times H_c^{5-\bullet}(\rS_\Gamma, \tm_{\lambda^\ast})  \longrightarrow \Q,$$ and $$H^\bullet(\partial \rS_\Gamma, \tm_\lambda) \times H^{4-\bullet}(\partial \rS_\Gamma, \tm_{\lambda^\ast})  \longrightarrow \Q.$$ These pairings are compatible with the restriction morphism $r:H^\bullet(\rS_\Gamma, \tm_\lambda) \longrightarrow H^\bullet(\partial \rS_\Gamma, \tm_\lambda)$  and the connecting homomorphism $\delta:H^\bullet(\partial \rS_\Gamma, \tm_\lambda) \longrightarrow H_c^{\bullet+1}(\rS_\Gamma, \tm_\lambda)$ of the long exact sequence in cohomology associated to the pair $(\rS, \partial \rS)$, in the sense that the pairings are compatible with the diagram :

$$\xymatrix{
 \ar@{->}^{r}[d] H^\bullet(\rS_\Gamma, \tm_\lambda) &\hspace{-1.25cm} \times &\hspace{-1.25cm}  H_c^{5-\bullet}(\rS_\Gamma, \tm_{\lambda^\ast})  \ar@{<-}^{\delta}[d] \ar@{->}[r] & \Q   \\
 H^\bullet(\partial \rS_\Gamma, \tm_\lambda) &\hspace{-1.25cm}  \times &\hspace{-1.25cm}  H^{4-\bullet}(\partial \rS_\Gamma, \tm_{\lambda^\ast}) \ar@{->}[r] & \Q }$$ 

$H_{Eis}^\bullet(\rS_\Gamma, \tm_\lambda)$ is the image of the restriction morphism $r$ and therefore, as an implication of the aforementionned compatibility between the pairings, the spaces $H_{Eis}^\bullet(\rS_\Gamma, \tm_\lambda)$ are maximal isotropic subspaces of the boundary cohomology under the Poincar\'e duality. This means that $H_{Eis}^\bullet(\rS_\Gamma, \tm_{\lambda^\ast})$ is the orthogonal space of $H_{Eis}^\bullet(\rS_\Gamma, \tm_{\lambda})$ under this duality.

In particular, one has 
\begin{equation}\label{eq:poincare}\mr{dim} H_{Eis}^\bullet(\rS_\Gamma, \widetilde{\mathcal{M}}_\lambda) + \mr{dim} H_{Eis}^\bullet(\rS_\Gamma, \widetilde{\mathcal{M}}_{\lambda^\ast}) = \frac{1}{2}\Big(\mr{dim} H^\bullet(\partial \rS_\Gamma, \widetilde{\mathcal{M}}_\lambda) + \mr{dim} H^\bullet(\partial \rS_\Gamma, \widetilde{\mathcal{M}}_{\lambda^\ast})\Big).\end{equation}

\subsection{Euler characteristic for boundary and Eisenstein cohomology}

In the next few lines we establish a relation between the homological Euler characteristics of the arithmetic group  and the Euler Characteristic of the Eisenstein cohomology of the arithmetic group, and similarly another relation with the Euler characteristic of the cohomology of the boundary. During this section we will be frequently using the notations $H_{Eis}^\bullet(\mr{SL}_3(\Z), \m_\lambda)$ for $H_{Eis}^\bullet(\rS_\Gamma, \widetilde{\mathcal{M}}_\lambda)$ and $H_{!}^\bullet(\mr{SL}_3(\Z), \m_\lambda)$ for $H_{!}^\bullet(\rS_\Gamma, \widetilde{\mathcal{M}}_\lambda)$ to make it very explicit the arithmetic group we are working with. See Section~\ref{Euler}, for the definition of homological Euler characteristic of $\mr{SL}_3(\Z)$. Note that we can define the ``naive" Euler characteristic of the underlying geometric object as the alternating sum of the dimension of its various cohomology spaces. Following this, we define 
$$ \chi(H_{Eis}^\bullet(\mr{SL}_3(\Z), \m_\lambda))=
\sum_i (-1)^i \dim H_{Eis}^\bullet(\mr{SL}_3(\Z), \m_\lambda)$$
and
$$ \chi(H^\bullet(\partial \rS, \tm_\lambda))=
\sum_i (-1)^i \dim H^\bullet(\partial \rS, \tm_\lambda).$$

The following two statements (Corollary~\ref{4point} and Lemma~\ref{keylemma}) are synthesized in Theorem~\ref{thm 1/2}, which is needed for computing the Eisenstein cohomology of $\mr{SL}_3(\Z)$ (see Theorem~\ref{Eiscoh}).  As a consequence of Theorem~\ref{bdcohsl3}, we obtain the following immediate 
\begin{coro}\label{4point}
\begin{equation*}
\chi(H^\bullet(\partial \rS, \tm_\lambda))= \left\{\begin{array}{cccc}  
& -2(1+\dim S_{m_1+2} +\dim S_{m_2+2} ) \,,  &  m_1, m_2 \,\, \mr{both}\, \mr{even}  \\ 
&\\
&  2(\dim S_{m_1+m_2+3} -\dim S_{m_1+2}) \,, & m_1 \, \mr{even},\, m_2 \, \mr{odd} \\
&\\
& 2(\dim S_{m_1+m_2+3} - \dim S_{m_2+2}) \,, & m_1 \, \mr{odd}, \, m_2 \, \mr{even}\\
&\\
& 0 \,, &  m_1, m_2 \, \,\mr{both\,\, odd}\\
\end{array}\quad\,, \right .  
 \end{equation*}
where we are denoting $dim S_2 = -1$.
\end{coro}

As discussed in the previous paragraph, we now state and prove a simple relation between Euler characteristic of the Eisenstein cohomology and the homological Euler characteristic.

\begin{lema}\label{keylemma}
\[\chi(H^\bullet_{Eis}(\mr{SL}_3(\Z),\m_\lambda))=\chi_h(\mr{SL}_3(\Z),\m_\lambda)\]
\end{lema}

\begin{proof} 
Let us denote by $h^i$,  $h^i_{Eis}$ and $h^i_{!}$ the dimension of the spaces $H^i(\mr{SL}_3(\Z),\m_\lambda)$,  $H^i_{Eis}(\mr{SL}_3(\Z),\m_\lambda)$ and $H^i_{!}(\mr{SL}_3(\Z),\m_\lambda)$, respectively. By definition, we have
\[\chi_h(\mr{SL}_3(\Z),\m_\lambda) =\sum_{i=0}^3(-1)^ih^i \,.\]
Assume $\lambda\neq 0$. Then $h^0=0$. Following Bass-Milnor-Serre, Corollary 16.4 in~\cite{BMS67}, we know that $h^1=0$. 

On the other hand, let $\m_{\lambda^\ast}$ be the dual representation of $\m_\lambda$. In our case, if $\lambda = (m_1 + m_2) \varepsilon_1 + m_2\varepsilon_2$, then $\lambda^\ast = (m_1 + m_2) \varepsilon_1 + m_1\varepsilon_2$. One has by Poincar\'e duality that $H^q_{!}(\mr{SL}_3(\Z),\m_\lambda)$ is dual to $H^{5-q}_{!}(\mr{SL}_3(\Z),\m_{\lambda^\ast})$. Moreover, if $\lambda \neq \lambda^\ast$ then $H^\bullet_{!}(\mr{SL}_3(\Z),\m_\lambda) = 0$ (see for example Lemma 3.2 of \cite{Harder2018}). Therefore one has, in all the cases, $h^2_{!}=h^3_{!}$.
Using that, we obtain
\begin{eqnarray*}
\chi_h(\mr{SL}_3(\Z),\m_\lambda)
&=&h^2-h^3
\\
&=&h^2_{Eis}
+
h^2_{!}
-
h^3_{Eis}
-
h^3_{!}
\\
&=&
h^2_{Eis}
-
h^3_{Eis}\\
&=&\chi(H^\bullet_{Eis}(\mr{SL}_3(\Z),\m_\lambda)).
\end{eqnarray*}
\end{proof}

We now state the following key result.
\begin{thm}
\label{thm 1/2}
\[
\chi(H^\bullet_{Eis}(\mr{SL}_3(\Z),\m_\lambda))
=
\frac{1}{2}\chi(H^\bullet(\partial \mr{S},\tm_\lambda)).
\]
\end{thm}
\begin{proof}
Using Corollary \ref{4point} and  Tables 8 and 9, we find that
$$
\chi(H^\bullet(\partial \mr{S},\tm_\lambda))
=
2\chi_h(\mr{SL}_3(\Z),\m_\lambda)).$$
Using Lemma \ref{keylemma}, we have
$$2\chi_h(\mr{SL}_3(\Z),\m_\lambda))
=
2\chi(H^\bullet_{Eis}(\mr{SL}_3(\Z),\m_\lambda)).
$$
Therefore,
$$
\chi(H^\bullet(\partial \mr{S},\tm_\lambda))
=2\chi(H^\bullet_{Eis}(\mr{SL}_3(\Z),\m_\lambda)).
$$
\end{proof}

\subsection{Main theorem on Eisenstein cohomology for $\mr{SL}_3(\Z)$}

The following is the main result of the paper, that gives both the dimension of the Eisenstein cohomology together with its sources - the corresponding parabolic subgroups. It is stated using different cases that cover all possible highest weight representations. A central part of the proof is based on Theorem~\ref{bdcohsl3} and Theorem~\ref{thm 1/2}.

\begin{thm}\label{Eiscoh}
$\,\,\,$
\begin{enumerate} 
\item Case 1 : $m_1 = m_2 = 0$  then 
\begin{equation*}
H^{q}_{Eis}(\mr{SL}_3(\Z), \m_\lambda) =\left\{\begin{array}{cccc}  
&\mathbb{Q} \quad \mr{for}\quad  q=0 \\ 
&\\
&0  \qquad \mr{otherwise}\\
\end{array}\qquad\,. \right .
\end{equation*}

\item  Case 2 :  $m_1=0$ and $m_2 \neq 0$, $m_2 $ even

\begin{eqnarray*}
H^{q}_{Eis}(\mr{SL}_3(\Z), \m_\lambda) 
& = &
\left\{\begin{array}{cccc}  
& S_{m_2+2}\,, & q=3 \\
&\\
&0  \,, & \mr{otherwise}\\
\end{array}\,. \right .
\end{eqnarray*}

\item  Case 3 :  $m_1 \neq 0$, $m_1$ even and $m_2 =0$

\begin{eqnarray*}
H^{q}_{Eis}(\mr{SL}_3(\Z), \m_\lambda) 
& = & 
\left\{\begin{array}{cccc}  
& S_{m_1+2} \,, & q=3 \\
&\\
&0  \,, & \mr{otherwise}\\
\end{array}\,. \right .
\end{eqnarray*}

\item  Case 4 :  $m_1 \neq 0$, $m_1$ even and $m_2 \neq 0$, $m_2$ even, then

\begin{eqnarray*}
H^{q}_{Eis}(\mr{SL}_3(\Z), \m_\lambda) 
& = &
\left\{\begin{array}{cccc}  
& \Q \oplus S_{m_1+2}\oplus S_{m_2+2} \,, & q=3 \\
&\\
&0  \,, & \mr{otherwise}\\
\end{array}\,. \right .
\end{eqnarray*}

\item   Case 5 :  $m_1 \neq 0$, $m_1$ even and $m_2$ odd, then

\begin{eqnarray*}
H^{q}_{Eis}(\mr{SL}_3(\Z), \m_\lambda) =
\left\{\begin{array}{cccc}  
& S_{m_1+m_2+3}\,, & q=2\\
&\\
& S_{m_1+2}\,, & q=3\\
&\\
& 0 \,, & \mr{otherwise}
\end{array} \,. \right.
 \end{eqnarray*}

\item  Case 6 :  $m_1 = 0$ and $m_2$ odd, then

\begin{eqnarray*}
H^{q}_{Eis}(\mr{SL}_3(\Z), \m_\lambda) 
& = & 
\left\{\begin{array}{cccc}
& S_{m_2+3} \oplus \Q \,, & q=2 \\
&\\
& 0 \,,& \mr{otherwise}
\end{array}\qquad\,, \right .\\
\end{eqnarray*}

\item  Case 7 :  $m_1 $ odd and  $m_2 =0$
\begin{eqnarray*}
H^{q}_{Eis}(\mr{SL}_3(\Z), \m_\lambda) 
& = & \left\{\begin{array}{cccc}
& S_{m_1+3} \oplus \Q \,, & q=2 \\
& \\
& 0 \,,& \mr{otherwise}
\end{array}\qquad\,. \right .
\end{eqnarray*}

\item  Case 8 :  $m_1$ odd and $m_2 \neq 0$, $m_2$ even, then
\begin{eqnarray*}\label{eq:520}
H^{q}_{Eis}(\mr{SL}_3(\Z), \m_\lambda) 
& = & \left\{\begin{array}{cccc}  
& S_{m_1+m_2+3} \,, & q=2\\
&\\
& S_{m_2+2}\,, & q=3\\
&\\
& 0 \,, & \mr{otherwise}
\end{array}\qquad\,, \right .\\
\end{eqnarray*}

\item  Case 9 :  $m_1$ odd and $m_2$ odd, then
\begin{equation*}
H^{q}_{Eis}(\mr{SL}_3(\Z), \m_\lambda) = 0 \,, \qquad \forall q \,.
\end{equation*}
\end{enumerate}
\end{thm}

\begin{proof}
Let \[ h^i=\dim H^i(\rm{SL}_3(\Z),{\mathcal M}_\lambda),\] \[h^i_! = \dim H^i_!(\rm{SL}_3(\Z),{\mathcal M}_\lambda),\] and \[ h^i_{Eis} := h^i_{Eis}(\widetilde{\mathcal M}_\lambda) = \dim H^i_{Eis}(\rm{SL}_3(\Z),\mathcal M_\lambda),\] \[h^i_{\partial} := h^i_{\partial}(\widetilde{\mathcal M}_\lambda) = \dim H^i(\partial \rS,\widetilde{\mathcal M}_\lambda).\]
For any nontrivial highest weight representation we have $h^0=0$, since any proper $\rm{SL}_3(\Z)$-invariant subrepresentation of ${\mathcal M}_\lambda$ is trivial. Also, $h^1=0$, from Bass-Milnor-Serre~\cite{BMS67}, Corollary 16.4. Therefore, $h^0_{Eis}=h^1_{Eis}=0$. Following~\cite{Soule78} and~\cite{BoSe73},
we know that the cohomological dimension of $\mr{SL}_3(\Z)$ is 3. 
Moreover, $h^2_!=h^3_!$ since the corresponding cohomology groups are dual to each other.
Therefore,
\begin{equation}
\label{eq:chiEis}
\chi_h(\mr{SL}_3(\Z),\m_\lambda)
=h^2-h^3=h^2_{Eis}-h^3_{Eis}.
\end{equation}
\subsubsection*{Cases 2, 3 and 4}
We have that 
$h^2_\partial=0$.
Therefore, 
$h^2_{Eis}=0$.
From equation~\eqref{eq:chiEis} and Theorem~\ref{thm 1/2}, we obtain
$h^3_{Eis}
=-\chi_h(\rm{SL}_3(\Z),{\mathcal M}_\lambda)
=-\frac{1}{2}\chi(H^\bullet(\partial\rS,\tm_\lambda)).$
Using Theorem 11, we conclude the formulas for case 2 and case 3 of Theorem 15.
\subsubsection*{Cases 6 and 7}
We have that 
$h^3_\partial=0$.
Therefore, 
$h^3_{Eis}=0$.
From equation~\eqref{eq:chiEis} and Theorem~\ref{thm 1/2}, we obtain
$h^2_{Eis}
=\chi_h(\mr{SL}_3(\Z),\m_\lambda)
=\frac{1}{2}\chi(H^\bullet(\partial \rS,\widetilde{\mathcal M}_\lambda)).$
Using Theorem~\ref{bdcohsl3}, we conclude the formulas for case 6 and case 7 of Theorem~\ref{Eiscoh}.
\subsubsection*{Cases 5 and 8}
The two cases are dual to each other. Thus it is enough to consider only case 5. 
From Poincar\'e Duality~\eqr{poincare}, we have
\begin{equation}
\label{eq:P}
\sum_i 
\left(
h^i_{Eis}(\widetilde{\mathcal M}_\lambda)
+h^i_{Eis}(\widetilde{\mathcal M}_{\lambda^*})
\right)
=
\frac{1}{2}
\sum_i 
\left(
h^i_{\partial}(\widetilde{\mathcal M}_\lambda)
+h^i_{\partial}(\widetilde{\mathcal M}_{\lambda^*})
\right)
\end{equation}
From Theorem
~\ref{thm 1/2}, we have
\begin{equation}
\label{eq:14}\sum_i 
(-1)^i \left(
h^i_{Eis}(\widetilde{\mathcal M}_\lambda)
+h^i_{Eis}(\widetilde{\mathcal M}_{\lambda^*})
\right)
=
\frac{1}{2}
\sum_i 
(-1)^i \left(
h^i_{\partial}(\widetilde{\mathcal M}_\lambda)
+h^i_{\partial}(\widetilde{\mathcal M}_{\lambda^*})
\right)
\end{equation}
Adding equations \eqref{eq:P} and \eqref{eq:14}, we obtain
\[
h^2_{Eis}(\widetilde{\mathcal M}_\lambda)
+h^2_{Eis}(\widetilde{\mathcal M}_{\lambda^*})
=
\frac{1}{2} 
\left(
h^2_{\partial}(\widetilde{\mathcal M}_\lambda)
+h^2_{\partial}(\widetilde{\mathcal M}_{\lambda^*}) \right).
\]
Subtracting equations~\eqref{eq:P} and~\eqref{eq:14}, we obtain
\begin{equation} \label{eq:odd89}
h^3_{Eis}(\widetilde{\mathcal M}_\lambda)
+h^3_{Eis}(\widetilde{\mathcal M}_{\lambda^*})
=
\frac{1}{2} 
\left(
h^3_{\partial}(\widetilde{\mathcal M}_\lambda)
+h^3_{\partial}(\widetilde{\mathcal M}_{\lambda^*}) \right) + \frac{1}{2} 
\left(
h^1_{\partial}(\widetilde{\mathcal M}_\lambda)
+h^1_{\partial}(\widetilde{\mathcal M}_{\lambda^*}) \right).
\end{equation}
Also, $\m_\lambda$ is a regular representation. Therefore,
\[
H^3_{Eis}(\mathrm{SL}_3(\Z),\mathcal{M}_\lambda) \subset
H^3(\partial \rS,\widetilde{\mathcal{M}}_\lambda)=S_{m_1+2}\,,\]
and
\[H^3_{Eis}(\mathrm{SL}_3(\Z),\mathcal{M}_{\lambda^*})\subset H^3(\partial \rS,\widetilde{\mathcal{M}}_{\lambda^*})=S_{m_2+2}.\]
Form, equation~\eqref{eq:odd89}, we have
\[
h^3_{Eis}(\widetilde{\mathcal{M}}_\lambda) + h^3_{Eis}(\widetilde{\mathcal{M}}_{\lambda^*})
=
\dim S_{m_1+2} + \dim S_{m_2+2}.
\]
Therefore, the above inclusions are equalities, \ie
\[H^3_{Eis}(\mathrm{SL}_3(\Z),\mathcal{M}_\lambda) = S_{m_1+2}\,, \quad   H^3_{Eis}(\mathrm{SL}_3(\Z),\mathcal{M}_{\lambda^*}) = S_{m_2+2}.
\]
Then 
\begin{equation}\label{eq:h2Eis}
h^2_{Eis}(\widetilde{\mathcal M}_\lambda)
=
\chi_h(\mathrm{SL}_3(\Z),{\mathcal M}_\lambda)
+h^3_{Eis}(\widetilde{\mathcal M}_\lambda)=
\dim S_{m_1+m_2+3}.
\end{equation}

Since 
$H^2_{Eis}(\mathrm{SL}_3(\Z),\mathcal M_{\lambda})
\subset
H^2(\partial \rS,\widetilde{\mathcal M}_{\lambda}),$ therefore
from Theorem~\ref{bdcohsl3} and equation~\eqr{h2Eis}, we conclude that
\[H^2_{Eis}(\mathrm{SL}_3(\Z),\m_{\lambda})
=
S_{m_1+m_2+3}.
\]
\end{proof}

Note that in case of $\mr{GL}_3(\Z)$, its highest weight representation $\m_{\lambda}$ is defined for highest weight $\lambda=m_1 \gamma_1 + m_2 \gamma_2 + m_3 \gamma_3$ with $\gamma_1 =\epsilon_1, \gamma_2 = \epsilon_1+ \epsilon_2, \gamma_3 = \epsilon_1+ \epsilon_2 +\epsilon_3$. In this case the cohomology groups $H^{q}(\rm{GL}(3,\Z),\m_\lambda)$ can be described explicitly which we state in the following lemma.

\begin{lema}\label{sl3togl3}
Let $\m_\lambda$ be the highest weight representation of $\mr{GL}_3(\Z)$ with $\lambda=m_1 \gamma_1 + m_2 \gamma_2 + m_3 \gamma_3$, then
{\small
\begin{equation*}
H^{q}(\mr{GL}_3(\Z), \m_{\lambda}) =\left\{\begin{array}{cccc}  
& 0 ,&  m_1+2m_2+3m_3 \equiv 1 (mod \, 2)  \\ 
&\\
&{{H^0(\mathbb{G}_m(\Z),H^{q}(\mr{SL}_3(\Z),\m_\nu))=H^{q} (\mr{SL}_3(\Z),\m_{\nu})}}, & m_1+2m_2+3m_3 \equiv 0 (mod \, 2) \\
\end{array}\qquad\,. \right.
\end{equation*}}
where $\m_\nu = \left.\m_{\lambda}\right|_{\mr{SL}_3}$, i.e. $\nu$ is the highest weight of $\mr{SL}_3$ given by $\nu=(m_1+m_2)\epsilon_1 + m_2\epsilon_2$.
\end{lema}

Note that the first equality is by Hochschild-Serre spectral sequence and the second one follows from the parity condition. Here $\mathbb{G}_m(\Z)=\{-1,1\}$. We may conclude the above discussion simply in the following corollary.
\begin{coro}\label{thm GL}
Let $\Gamma$ be either $\mr{GL}_3(\Z)$ or $\mr{SL}_3(\Z)$, and $\m_{\lambda}$ be any highest weight representation of $\Gamma$. The following are true.
\begin{enumerate}
\item If $\m_{\lambda}$ is not self dual then \[H^{q}(\Gamma, \m_{\lambda}) = H^{q}_{Eis}(\Gamma, \m_{\lambda}).\]
\item If $\m_{\lambda}$ is  self dual then we have \[H^{q}(\Gamma, \m_{\lambda})= H^{q}_{Eis}(\Gamma, \m_{\lambda}) \oplus H^{q}_{!}(\Gamma, \m_{\lambda}),\]
where $H^{2}_{!}(\Gamma, \m_{\lambda})$ and $H^{3}_{!}(\Gamma, \m_{\lambda})$ are dual to each other, and $H^{0}_{!}(\Gamma, \m_{\lambda})=H^{1}_{!}(\Gamma, \m_{\lambda})=0$.
\end{enumerate}
\end{coro}

\begin{rmk}
In Theorem~\ref{Eiscoh} and hence in Corollary~\ref{thm GL}, we obtain exactly the dimensions of the group cohomlogy $H^i(\mr{GL}_3(\Z),\m_\lambda)$ and $H^i(\mr{SL}_3(\Z),\m_\lambda)$, when the highest weight representation $\m_\lambda$ is not self dual. For self dual representations, the result gives lower bounds for the dimensions because the discrepancy between the total cohomology and the Eisenstein cohomology is the inner cohomology (which over $\mathbb{C}$ contains the cuspidal cohomology) that is nonzero only in degrees $2$ and $3$. Even more, because of Poincar\'e duality, the inner cohomology in degree $2$ is dual to the inner cohomology in degree $3$.
\end{rmk}




\section{Ghost Classes}\label{ghost} 

Following the discussion in Section~\ref{basic}, we have 
\begin{equation*} \label{eq:first}
\ldots \rightarrow H_c^q(\rS, \tm_\lambda) \longrightarrow H^q(\rS, \tm_\lambda) \xrightarrow{r^q} H^q(\partial \overline{\rS}, \tm_\lambda) \longrightarrow \ldots
\end{equation*}  
 and the covering $\partial \overline{\rS} = \cup_{\rP \in \mathcal{P}_\mathbb{Q}(G)}\partial_{\rP}$,  which induces a spectral sequence in cohomology connecting to $H^\bullet(\partial \overline{\rS}, \tm_\lambda)$, leads to another long exact sequence in cohomology
\begin{equation} \label{eq:second}
\ldots \longrightarrow H^q(\partial \overline{\rS}, \tm_\lambda) \xrightarrow{p^q} H^q(\partial_{1}, \tm_\lambda) \oplus H^q(\partial_{2}, \tm_\lambda) \longrightarrow H^q(\partial_{0}, \tm_\lambda) \longrightarrow \ldots
\end{equation} 
We now define the space of $q$-ghost classes by $$Gh^q(\tm_\lambda) = Im(r^q) \cap Ker(p^q).$$ 

\par We will see that for almost every $q$ and $\lambda$, $Gh^q(\tm_\lambda) = \{0\}$. For pedagogical reasons, we now provide the details for all the nine cases. To begin with let us define  the maps $$s_{_q} : H^{q-1}(\partial_0, \tm_{\lambda}) \longrightarrow H^{q}(\partial \rS, \tm_{\lambda}) $$  and  for $i=1,2$ $$r^{q}_{i} :  H^{q}(\partial_i, \tm_{\lambda}) \longrightarrow H^{q}(\partial_0, \tm_{\lambda}) .$$ Note that $H^{q}(\rS, \tm_{\lambda}) =0$ for $q =1$ and $q \geq 4$.  Following this in all the cases, we obtain $Im(r^{q})= \{0\}$ for $q = 1$ and $q \geq 4.$  Also, in every case,  $Ker(p^{0}) = Im(s_{_{0}}) = \{0\}$. Therefore, it is easy to see that in all the cases we get the following conclusion.
\begin{lema}\label{q=014} 
For any highest weight $\lambda$, $Gh^q(\tm_\lambda) = \{0\}$ for $q = 0,1$ and $q \geq 4.$
\end{lema}
Now, what remains to discuss is the space $Gh^{q}(\tm_{\lambda})$ for $q=2,3$. Following the above discussion, we observe that in case 1 and case 9,   $Gh^{q}(\tm_{\lambda})=\{ 0\} , \forall q$. Since from Theorem~\ref{bdcohsl3}, $H^{q}(\partial \rS, \tm_{\lambda}) =0$, $\forall q$ in case 9 and for $q=1,2,3$ in case 1.  

\par Note that case 2 and case 3,  are dual to each other. We know that $H^{2}(\partial \rS, \tm_{\lambda})=0$ therefore $Ker(p^{2})=\{0\}$. This gives us $Gh^{2}(\tm_{\lambda}) = \{ 0\}.$ For $q=3$ we have $$Gh^{3}(\tm_{\lambda}) = Im(r^3) \cap Im(s_{_{3}}) ,$$ where 
$s_{_{3}} : H^{2}(\partial_0, \tm_{\lambda}) \longrightarrow H^{3}(\partial S, \tm_{\lambda}) ,$
and following~\eqr{e211q} we see that $Im(s_{_{3}})=\{ 0\}$ since $H^{2}(\partial_0, \tm_{\lambda})=0$. In other words, in case 2, there are no second degree cohomology classes of $\rP_0$ and this implies that the domain of $s_{_{3}}$ is zero. Hence, the image is so. We conclude this in the form of following lemma.
\begin{lema}
In case 2 and case 3, \ie for $\lambda=m_2 \gamma_2$ and $\lambda=m_1 \gamma_1 $, respectively, with $m_1, m_2$ non zero even integers,  $Gh^q(\tm_\lambda) = \{0\}, \, \, \forall q .$
\end{lema}
Let us discuss now the case 6 and case 7. Following Theorem~\ref{bdcohsl3}, $H^{3}(\partial \mr{S}, \tm_{\lambda})=0$ and therefore $Gh^3(\tm_\lambda) = \{0\}$. By the definition of ghost classes, we have
$Gh^2(\tm_\lambda)=Im(r^2) \cap Im(s_{_{2}}) $ where $ s_{_{2}} : H^{1}(\partial_0, \tm_{\lambda}) \longrightarrow H^{2}(\partial \rS, \tm_{\lambda}) ,$ \ie
\begin{equation*}\label{eq:s12} s_{_{2}} :  H^{0}(\rS^{\rm{M}_{0}}, \widetilde{\m}_{s_1\cdot\lambda}) \oplus H^{0}(\rS^{\rm{M}_{0}}, \widetilde{\m}_{s_2\cdot\lambda})\longrightarrow  H_{!}^{1}(\rS^{\rm{M}_1}, \tm_{s_1\cdot\lambda}) \oplus H_{!}^{1}(\rS^{\rm{M}_2}, \tm_{s_2\cdot\lambda}) \oplus \Q \oplus \Q.\end{equation*}
However,
$ H^{0}(\rS^{\rm{M}_{0}}, \widetilde{\m}_{s_2\cdot\lambda})=0$
and
$\dim  H^{0}(\rS^{\rm{M}_{0}}, \widetilde{\m}_{s_1\cdot\lambda})=1.$
Therefore, in case 6 and case 7, either
$\dim Gh^2(\widetilde{M}_\lambda)=0$ or $1$.
\begin{lema}
In case 6 and case 7, \ie for $\lambda=m_2 \gamma_2$ and $\gamma=m_1\gamma_1$, respectively, with $m_1$ and $m_2$ any odd integer, $Gh^q(\tm_\lambda) = \{0\}, \, \, \forall q,$ except possibly for $q =2$.
\end{lema}
Consider now the case 5 and case 8. In case 5, $Gh^2(\tm_\lambda) =\{ 0\} $ since $Ker(p^{2})=\{0\}$. This simply follows by studying $Im(s_{_2})$ where $s_{_2}$ is defined by
$$ s_{_{2}} :  H^{0}(\rS^{\rm{M}_{0}}, \widetilde{\m}_{s_1\cdot\lambda}) \longrightarrow  H_{!}^{1}(\rS^{\rm{M}_1}, \tm_{s_1\cdot\lambda}) \oplus H_{!}^{1}(\rS^{\rm{M}_2}, \tm_{s_2\cdot\lambda}) \,,$$ and $Ker(s_2)$ is the image of the morphism
\begin{equation*}
H^1(\partial_{1}, \tm_\lambda) \oplus H^1(\partial_{2}, \tm_\lambda) \longrightarrow H^1(\partial_{0}, \tm_\lambda)
\end{equation*}
from the exact sequence ~\eqr{second}. From the calculations in Section ~\ref{bdsl3} we get $Im(s_{_2})=0$. Similarly, we have
\begin{eqnarray*}
s_{_{3}} :  H^{0}(\rS^{\rm{M}_{0}}, \widetilde{\m}_{s_1s_2\cdot\lambda}) \oplus H^{0}(\rS^{\rm{M}_{0}}, \widetilde{\m}_{s_2s_1\cdot\lambda})  \longrightarrow & H_{!}^{1}(\rS^{\rm{M}_1}, \tm_{s_1s_2\cdot\lambda})
\cong & 
S_{m_2+2} .
\end{eqnarray*}
and again by same reasoning, we see that $s_{_3}$ vanishes. Therefore $Gh^3(\tm_\lambda) =\{0\}$. Case 8 is analogous and we simply conclude the following.
\begin{lema}
In case 5 and case 8, \ie for $\lambda=m_1 \gamma_1+m_2 \gamma_2$ with $m_1$ and $m_2$ nonzero and having different parity modulo $2$,  $Gh^q(\tm_\lambda) = \{0\}, \, \, \forall q .$
\end{lema}
The only case that remains to discuss is case 4. Following Lemma~\ref{q=014}, the only cases which need to be discussed are $q=2$ and $q=3$. However, following case 4 of Theorem~\ref{bdcohsl3}, we know that $H^{2}(\partial S, \tm_{\lambda}) =0$, therefore $Gh^2(\tm_\lambda) = \{0\}.$ $Gh^3(\tm_\lambda) = 0$ because $H^2(\partial_{0},\widetilde{\mathcal M}_\lambda) = 0$. Hence, we can simply summarize this in the form of following lemma.
\begin{lema}
In case 4, \ie for $\lambda=m_1 \gamma_1+m_2 \gamma_2$ with $m_1, m_2$ both non zero even integers,  $Gh^q(\tm_\lambda) = \{0\}, \, \, \forall q .$
\end{lema}
\begin{rmk}  
We can summarize the whole discussion of this section in the following lines to give the reader an intuitive idea of how to get to the punchline. The kernel of $p^q$ is isomorphic to the image of $s_{_{q}}$ and the image of $r^q$ is the  Eisenstein cohomology of degree $q$. Thus the ghost classes are classes in the Eisenstein cohomology that are also in the image of the connecting homomorphism $s_q$. Since the Eisenstein cohomology is concentrated in degrees $2$ and $3$, see Theorem~\ref{Eiscoh}, we have that any ghost class of $\mr{SL}_3(\Z)$ must come from the image of $H^1(\partial_{0},\widetilde{\mathcal M}_\lambda)$ or  $H^2(\partial_{0},\widetilde{\mathcal M}_\lambda)$ in $H^2(\partial \rS,\widetilde{\mathcal M}_\lambda)$ or $H^3(\partial \rS,\widetilde{\mathcal M}_\lambda)$, respectively. Examining all the nine cases of boundary cohomology (see Theorem~\ref{bdcohsl3}), we see that there is no contribution from the minimal parabolic subgroup $\rP_0$ to the boundary cohomology of degree $2$ or $3$, except in the cases 6 and 7. Thus, there are no ghost classes in $\mr{SL}_3(\Z)$ and similarly in $\mr{GL}_3(\Z)$, except possibly in the cases 6 and 7. 
\end{rmk}
Hence, we summarize the discussion in the following theorem.
\begin{thm}\label{ghostthm}There are no nontrivial ghost classes in $\mr{SL}_3(\Z)$ and $\mr{GL}_3(\Z)$, except in the cases 6 and 7. In those cases, non-zero ghost classes might occur only in degree 2, where we have
$Gh^2(\widetilde{M}_\lambda)=0$ or $\Q$.
\end{thm}
\section*{Acknowledgement}
The authors would like to thank the Max Planck Institute for Mathematics (MPIM), Bonn, where most of the discussion and work took place, for its hospitality and support. 

JB would like to thank the Mathematics Department of the Georg-August University G\"ottingen for the support, and especially to Valentin Blomer and Harald Helfgott for their encouragement during the writing of this article. In addition, JB would like to thank the organizers of the HIM trimester program on ``Periods in Number Theory, Algebraic Geometry and Physics" for giving him the opportunity to participate, where he benefited through many stimulating conversations with G\"unter Harder and  Madhav Nori, and extend his thanks to the Institut des Hautes \'Etudes Scientifiques (IHES), Paris, for their hospitality during the final work on this article. JB's work is financially supported by ERC Consolidator grant 648329 (GRANT). 

IH would like to thank G\"unter Harder for the multiple discussions during his visit to MPIM, and Karen Vogtmann and Martin Kassabov, for raising an important question which has been answered through the work of this article.

 MM would like to thank IHES, Universit\'e Paris 13 and Universit\'e Paris-Est Marne-la-Vall\'ee  for their hospitality, and G\"unter Harder for his support and for the many inspiring discussions.
 
\nocite{}
\bibliographystyle{abbrv}
\bibliography{EC}
\end{document}